\documentclass[11pt,reqno, english]{amsart}

\usepackage{amsfonts}
\usepackage{amssymb}
\usepackage{amsthm}
\usepackage{amsmath}
\usepackage[top=4cm, bottom=3.5cm, left=2cm, right=2cm,marginparwidth=1.8cm, marginparsep=0.1cm]{geometry}

\newtheorem{prop}{Proposition}[section]
\newtheorem{theorem}[prop]{Theorem}
\newtheorem{lemma}[prop]{Lemma}
\newtheorem{corollary}[prop]{Corollary}

\newtheorem{fact}[prop]{Fact}

\theoremstyle{definition}
\newtheorem{definition}[prop]{Definition}

\theoremstyle{remark}
\newtheorem{remark}[prop]{Remark}

\newtheorem{example}[prop]{Example}

\newtheorem*{remark*}{Remark}
\newtheorem*{remarks*}{Remarks}

\numberwithin{equation}{section}

\newcommand{\R}{\mathbb{R}}

\newcommand{\N}{\mathbb{N}}
\newcommand{\Z}{\mathbb{Z}}

\newcommand{\F}{\mathbb{F}}

\newcommand{\eps}{\varepsilon}

\newcommand{\Aut}{\operatorname{Aut}}
\newcommand{\diam}{\operatorname{diam}}

\begin{document}

\title{Nilprogressions and groups with moderate growth}
\date{}

\author{Emmanuel Breuillard and Matthew C. H. Tointon}


\address{Laboratoire de Math\'ematiques d'Orsay, Univ.~Paris-Sud, CNRS, Universit\'e Paris-Saclay, 91405 Orsay, France}
\email{emmanuel.breuillard@math.u-psud.fr}

\address{Laboratoire de Math\'ematiques d'Orsay, Univ.~Paris-Sud, CNRS, Universit\'e Paris-Saclay, 91405 Orsay, France}
\email{matthew.tointon@math.u-psud.fr}

\begin{abstract}
We show that doubling at some large scale in a Cayley graph implies uniform doubling at all subsequent scales. The proof is based on the structure theorem for approximate subgroups proved by Green, Tao and the first author. We also give a number of applications to the geometry and spectrum of finite Cayley graphs. For example, we show that a finite group has moderate growth in the sense of Diaconis and Saloff-Coste if and only if its diameter is larger than a fixed power of the cardinality of the group. We call such groups almost flat and show that they have a subgroup of bounded index admitting a cyclic quotient of comparable diameter. We also give bounds on the Cheeger constant, first eigenvalue of the Laplacian, and mixing time. This can be seen as a finite-group version of Gromov's theorem on groups with polynomial growth. It also improves on a result of Lackenby regarding property $(\tau)$ in towers of coverings. Another consequence is a universal upper bound on the diameter of all finite simple groups, independent of the CFSG.
\end{abstract}

\maketitle
\setcounter{tocdepth}{1}
\tableofcontents

\section{Introduction}

Let $G$ be a group generated by a finite, symmetric subset $S$. Here, and throughout this paper, by writing that $S$ is \emph{symmetric} we mean that if $s$ belongs to $S$ then so does $s^{-1}$, \emph{and} that $S$ contains the identity. When speaking of the growth of $G$ with respect to $S$, we refer to the behaviour of the sequence of cardinalities
\[
|S|,|S^2|,|S^3|,\ldots,
\]
where we denote by $S^n$ the $n$-fold product set $\{s_1\cdot ... \cdot s_n ; s_i \in S\}$. This is the ball of radius $n$ in the Cayley graph of $G$ relative to $S$. In the event that $G$ is finite, we also consider the \emph{diameter} $\diam_S(G)$ of $G$ with respect to $S$, which is defined to be the minimum $n$ such that $S^n=G$.

According to Gromov's polynomial growth theorem \cite{gromov}, if the sequence $\{|S^n|\}$ is bounded above by a polynomial function of $n$, then $G$ has a finite-index nilpotent subgroup. In this paper we are concerned with some refinements of this theorem, mostly in the context of finite groups. In particular, we study the relations between the diameter and the cardinality of a finite group on the one hand, and various interesting invariants such as the Cheeger constant, the first non-zero eigenvalue of the Laplace operator, and the mixing time of the associated Cayley graph on the other.

Our two main results are  Theorems \ref{thm:scales} and \ref{thm:gromovfinite} below. The first concerns the doubling property at one given scale in an arbitrary Cayley graph.

\begin{theorem}[Doubling at some scale implies doubling at all larger scales]\label{thm:scales}
For every $K\ge1$ there exist $n_0=n_0(K)\in\N$ and $\theta(K) \geq 1$, such that if $S$ is a finite symmetric set inside some group, and if there exists $n\ge n_0$ for which
\begin{equation}\label{doubling}
|S^{2n+1}|\le K|S^n|,
\end{equation}
then for every $m\ge n$ and every $c\in \N$ we have $|S^{cm}|\leq \theta(K)^c  |S^m|$.
\end{theorem}

We stress that the constants $n_0(K), \theta(K)$ depend only on $K$, and not on the group or on $S$. Most of our results in this paper are uniform in the generating set $S$ and do not assume that $S$ is bounded.  Unfortunately the method does not allow for an effective computation of the constants $n_0(K)$ and $\theta(K)$ in terms of $K$.

The doubling property played an important role in Gromov's original proof of his polynomial-growth theorem. Indeed, the first step of his proof consisted of observing that the polynomial growth condition on $|S^n|$ implies that a doubling condition such as $(\ref{doubling})$ holds for infinitely many $n$ and some uniform $K$. Green, Tao and the first author established in \cite{bgt} a structure theorem for doubling sets (or equivalently approximate subgroups; see Section \ref{app-sec}) in arbitrary groups, allowing them to extend Gromov's theorem in a number of ways. The proof of Theorem \ref{thm:scales} makes crucial use of this structure theorem via an analysis of the behaviour of large powers of nilprogressions (see Section \ref{sec:coms.nilprogs}, below, for the definition of nilprogressions).

\begin{remark*}Tao \cite{tao.gromov} has independently used a somewhat similar argument to classify fairly explicitly the possible behaviours of $|S^m|$ as $m\to\infty$ under the hypothesis \eqref{doubling}, in particular recovering Theorem \ref{thm:scales}.
\end{remark*}

\bigskip

The motivation for Theorem \ref{thm:scales} came from work of Benjamini, Finucane and Tessera \cite{bft}, in which the authors identify the scaling limit of a sequence of vertex-transitive graphs of large diameter (i.e. such that $(\ref{diambnd})$ below holds).
Using Theorem \ref{thm:scales} they show\footnote{Their original argument required only the main result of \cite{bgt}, not Theorem \ref{thm:scales}, but needed to assume a bound on $|S|$.}  that for a given $\eps>0$, the family of all such Cayley graphs (rescaled so that edges have length $1/\diam_S(G)$) is relatively compact for the Gromov-Hausdorff topology on the space of compact metric spaces. Indeed a doubling condition such as that of Theorem \ref{thm:scales} is exactly what is needed to apply Gromov's relative compactness criterion \cite{gromov-book}. Furthermore, they identify the possible limits as the flat tori $\R^d/\Z^d$, $d=d(\eps)$, endowed with a translation invariant Finsler metric.

This prompts the following definition.

\begin{definition}Let $\eps>0$. A Cayley graph of a finite group $G$ with symmetric generating set $S$ is called \emph{$\eps$-almost flat} if
\begin{equation}\label{diambnd}
\diam_S(G)\ge\left(\frac{|G|}{|S|}\right)^\varepsilon.
\end{equation}
\end{definition}

See \cite{bft} and the survey \cite{minessota-survey} for more details on the result of Benjamini, Finucane and Tessera. We will give further applications of Theorem \ref{thm:scales} below.

\bigskip

\noindent \emph{Remark.} Abelian groups and more generally nilpotent groups with a bound on the nilpotency class and number of generators are $\eps$-almost flat for some positive $\eps$ depending only on these bounds. See Proposition \ref{prop:ab.converse}, below. We also show in Proposition \ref{cor:intrinsic} that the notion of an almost flat group is independent of the choice of generating set, provided the size of the generating set remains bounded.

\bigskip

Our second main result concerns the algebraic structure of almost flat finite groups, and runs as follows.

\begin{theorem}[Almost flat groups have large virtually abelian quotients]\label{thm:gromovfinite} Let $\eps>0$ and $\beta>\frac{1}{2}$. Every $\eps$-almost flat Cayley graph of a finite group $G$ with generating set $S$ has a normal subgroup $H$ contained in $S^{O_{\eps,\beta}(\gamma^{\beta})}$ such that $G/H$ has an abelian subgroup whose rank and index are bounded in terms of $\eps$ and $\beta$ only.
\end{theorem}

The implied constants depend only on $\eps$ and $\beta$, and not on the group $G$ or the size of the generating set $S$. We refer the reader to Theorem \ref{struct} below for two variants of Theorem \ref{thm:gromovfinite}. In particular, we show there that $G$ must have a subgroup of bounded index with a cyclic quotient of comparable diameter.

We note that Gromov had already proved a finite version of his polynomial growth theorem (see at the end of \cite{gromov}). The above can be seen as a refinement, in which a weaker assumption is made (a diameter condition instead of a volume growth condition) and a somewhat stronger conclusion is derived.

\bigskip

It is legitimate to ask whether the conclusion of Theorem \ref{thm:gromovfinite} (or those of Theorem \ref{struct} below) continues to hold under a weaker assumption on the diameter. The answer is no. In fact, given any function $f:\N \to \R_+$ going to zero at infinity, there is a sequence of pairwise non-isomorphic finite groups $G_n$ with uniformly bounded generating sets with respect to which their diameters are at least $|G_n|^{f(|G_n|)}$, while $G_n$ has no proper subgroup of index at most $n$ and no non-trivial abelian quotient. This example also provides a counter-example to a non-commutative version of the \emph{polynomial Freiman-Ruzsa} conjecture. See Subsection \ref{sharpness} below.

\bigskip

We now present some applications of Theorems \ref{thm:scales} and \ref{thm:gromovfinite}.

\subsection{Diameter of finite simple groups}

The following is an immediate consequence of Theorem \ref{thm:gromovfinite}.

\begin{corollary}[Diameter bound for finite simple groups]\label{cor:fsg}
Let $\varepsilon>0$. Then there is a constant $C_\eps$ depending only on $\eps$ such that every non-abelian finite simple group $G$ with symmetric generating set $S$ satisfies
\[
\diam_S(G)\le\max\left\{\left(\frac{|G|}{|S|}\right)^\varepsilon,C_\eps\right\}.
\]
\end{corollary}

\begin{proof} Write $\gamma=\diam_S(G)$. Since $G$ is simple, if $\gamma>(|G|/|S|)^\eps$ then Theorem \ref{thm:gromovfinite} implies either that $\gamma\le O_{\eps}(\gamma^{3/4})$, in which case $\gamma\le O_{\eps}(1)$, or that $G$ has a normal abelian subgroup of index at most $O_{\eps}(1)$. This normal subgroup is either all of $G$, in which case $G$ is abelian, or it is trivial, in which case $|G|\le O_{\eps}(1)$, and so certainly $\gamma\le O_{\eps}(1)$.
\end{proof}

This bound should be compared with the conjecture of Babai (see \cite{babai}) according to which there ought to exist universal constants $C,D>0$ such that $\diam_S(G)\leq D (\log |G|)^C$ for all finite simple groups $G$ and all $S$. Such a strong bound is currently only known for finite simple groups of bounded rank (see \cite{bgt1,helfgott, pyber-szabo}). In fact for these groups it is even conjectured that one can take $C=1$ with a multiplicative constant $D$ that is allowed to depend on the rank \cite[Conjecture 4.5]{breuillard-icm}. For alternating groups $A_n$, weaker bounds that are still much better than that of Corollary \ref{cor:fsg}, are known (see \cite{babai, helfgott-seress}).

Corollary \ref{cor:fsg} is the first non-trivial bound towards Babai's conjecture valid for all finite simple groups as far as we can tell. However, even if it is not straightforward, it is likely that a better bound also valid for all finite simple groups is within reach by current methods using the classification of finite simple groups and the above mentioned results. Let us note however that a remarkable feature of the proof of Corollary \ref{cor:fsg} is that it does not use the classification of finite simple groups.

\bigskip

\subsection{Cheeger constant and spectral gap}

It is well known that the diameter of a Cayley graph is related to its Cheeger constant $h$ and to the first eigenvalue $\lambda_1$ of the Laplace operator in such a way that
\begin{eqnarray}\label{groupbounds}
\frac{1}{8 (\diam_S(G))^2}\leq \frac{1}{2}{h^2} \leq \lambda_1 \leq 2|S|h \leq \frac{ 8|S|^2 \log |G|}{\diam_S(G)}
\end{eqnarray}
See Section \ref{expanderbounds} for a definition of these quantities and a proof of \eqref{groupbounds}.

\begin{corollary}[Groups with very small Cheeger constant or  spectral gap]\label{smallcheeger} Let $G$ be a finite group with finite symmetric generating set $S$. Let $\eps>0$. Assume that the Cayley graph of $G$ with respect to $S$ satisfies one of the inequalities
$$\lambda_1 \leq \min\left\{2^{-9},\left(\frac{|S|}{|G|}\right)^\eps\right\} ,\qquad\qquad  h \leq \min\left\{2^{-2},\left(\frac{|S|}{|G|}\right)^\varepsilon\right\}. $$ Then $G$ is $\eps/3$-almost flat. In particular, the conclusion of Theorem \ref{thm:gromovfinite} holds.
\end{corollary}

This is useful in a variety of contexts, for example in the work of Ellenberg, Hall and Kowalski \cite{ehk} in arithmetic geometry, where they used a weak lower bound on $\lambda_1$ for certain perfect groups due to Pyber and Szab\'o \cite{pyber-szabo}. The bound from Corollary \ref{smallcheeger} would work there equally well. Note that, conversely, it is easy to verify that if a finite group $G$ has subgroup of bounded index with a cyclic quotient of comparable diameter then its $\lambda_1$ is very small, indeed of order $O_{|S|}(1)/(\diam_S(G))^2$.

\begin{corollary}[Cheeger constant and spectral gap for almost flat groups]\label{cheeger-gap} Given $\eps>0$ there are $C_1,C_2>0$ such that if $G$ is a finite group with $\eps$-almost flat Cayley graph, then its first eigenvalue $\lambda_1$ and Cheeger constant $h$ satisfy
\begin{equation}\label{eq:ch-gap.1}
\frac{1}{8 (\diam_S(G))^2}\leq \frac{1}{2}{h^2} \leq \lambda_1 \leq C_1 |S| h^2 \leq \frac{C_2 |S|^2}{ (\diam_S(G))^2}.
\end{equation}
\vspace{2pt}

\noindent In fact, there is $C_3>0$ such that
\begin{equation}\label{eq:ch-gap.2}
\lambda_1 \leq \frac{C_3 |S|}{ (\diam_S(G))^2}.
\end{equation}
\end{corollary}

In other words the Buser inequality holds in its strong form on almost flat groups, which is to say that $\lambda_1$ is comparable to $h^2$. This follows from a simple argument using the doubling property at a scale comparable to the diameter. See Lemma \ref{lambdadoubling}. It is known by work of Klartag et al. \cite{klartag-kozma} that for abelian $G$ we may take $C_1=16$, which is essentially sharp.

\bigskip

\subsection{Property $(\tau)$ and towers of coverings}
A finitely generated group $\Gamma$ is said to have property $(\tau)$ with respect to a sequence of finite-index normal subgroups $\{\Gamma_n\}_n$ if the Cheeger constant $h(\Gamma/\Gamma_n)$ of the Cayley graph of the finite group $\Gamma/\Gamma_n$ with respect to a fixed generating set of $\Gamma$ is bounded away from zero, independently of $n$. If $\Gamma$  has property $(\tau)$ with respect to the family of all its finite-index normal subgroups then one simply says that $\Gamma$ has property $(\tau)$. This property is independent of the choice of generating set and stable under passing to a finite-index subgroup. Property $(\tau)$ was introduced and studied by Lubotzky; see \cite{breuillard-pcmi,lubotzky-book, lubotzky-weiss} for background.

It is clear that a finitely generated group that admits a finite-index subgroup with infinite abelianisation cannot have property $(\tau)$, as it will have large virtually cyclic quotients. The converse is not true, because of the existence of amenable finitely generated infinite torsion groups, such as the Grigorchuk group (see e.g. \cite{lubotzky-weiss}). However, it is not known whether or not the converse holds in the category of finitely presented groups.

In connection with Thurston's virtual first Betti number conjecture regarding the topology of $3$-manifolds, Lackenby \cite{lackenby-israel, lackenby-algebra, lack} studied various extra conditions of an algebraic or topological nature on a finitely presented group without property $(\tau)$ that force the existence of a finite-index subgroup with infinite abelianisation. See also \cite{lack.icm}. In particular, he showed that if the Cheeger constant $h(\Gamma/\Gamma_n)$ decays faster than the square root of $|\Gamma/\Gamma_n|$, then $\Gamma$ has a finite-index subgroup with infinite abelianisation \cite[Theorem 1.1]{lack}. As a consequence of Theorem \ref{thm:gromovfinite}, we can remove the finite presentability assumption and strengthen Lackenby's result as follows.

\begin{corollary}[Strong failure of $(\tau)$ implies positive virtual Betti number -- group form]\label{cover-group} Let $\Gamma$ be a finitely generated group, and $\{\Gamma_n\}_n$ an infinite family of finite-index normal subgroups such that $h(\Gamma/\Gamma_n)[\Gamma:\Gamma_n]^{\eps}$ is bounded independently of $n$, for some fixed $\eps>0$. Then $\Gamma$ has a finite-index subgroup with a surjective homomorphism onto $\Z$.
\end{corollary}

Using the well-known transfer principle between spectral properties of finite covers of Riemannian manifolds and those of the Cayley graphs of the associated groups of deck transformations (see \cite{brooks}, \cite{mantuano}, \cite[Appendix]{breuillard-pcmi}), one has the following geometric consequence of Corollary \ref{cover-group}.

\begin{corollary}[Strong failure of $(\tau)$ implies positive virtual Betti number -- geometric form]\label{cover-manifold} Let $M$ be a compact Riemannian manifold and $M_n$ a sequence of finite Galois covers of $M$ of degree $d_n$. Assume that there is $\eps>0$ such that $h(M_n)(d_n)^{\eps}$ is a bounded sequence, where $h(M_n)$ denotes the Cheeger constant of $M_n$. Then $M$ has a finite cover with positive first Betti number.
\end{corollary}

Note that we could have stated these corollaries with $\lambda_1$ in place of $h$, because of the Cheeger-Buser inequalities $(\ref{groupbounds})$. We take this opportunity to raise the following question.

\bigskip

\noindent \emph{Question:} Can the assumption that the covers are Galois in Corollary \ref{cover-manifold}, or the assumption that $\Gamma_n$ is normal in Corollary \ref{cover-group}, be removed? Suppose even that the (not necessarily normal) finite-index subgroups $\Gamma_n$ are such that the diameter of the Schreier graph of $\Gamma/\Gamma_n$ is at least $[\Gamma,\Gamma_n]^\eps$, for some fixed $\eps>0$. Does this imply that $\Gamma$ virtually maps onto $\Z$?

\bigskip

\subsection{Groups with moderate growth and mixing times}
Given positive constants $A$ and $d$,  Diaconis and Saloff-Coste \cite{dsc} say that the finite group $G$ has \emph{$(A,d)$-moderate growth} with respect to a finite generating set $S$ if, writing $\gamma:=\diam_S(G)$ for the diameter of $G$ with respect to $S$, we have
\begin{equation}\label{moderatedef}
|S^n|\ge\frac{1}{A}\left(\frac{n}{\gamma}\right)^d|G|
\end{equation}
whenever $1\le n\le\gamma$. It is clear from the definition, taking $n=1$, that $(G,S)$ is $\frac{1}{2d}$-almost flat, provided $|G|/|S|$ is large enough (in fact this holds as soon as the diameter $\diam_S(G)$ is larger that some constant depending only on $A$ and $d$). In particular, by Theorem  \ref{thm:gromovfinite} such a group has a large virtually cyclic quotient. Using Theorem \ref{thm:scales} we can prove that, conversely, every $\eps$-almost flat group is of moderate growth.

\begin{corollary}[Groups with moderate growth  versus almost flat groups]\label{cor:diac.sc}
A group has moderate growth if and only if it is almost flat. More precisely, given $A,d>0$ there is $C(A,d)>0$ such that every finite group with $(A,d)$-moderate growth and diameter at least $C(A,d)$ is $\frac{1}{2d}$-almost flat. Conversely, given $\eps>0$ there are $A,d>0$ depending only on $\eps$, such that every $\eps$-almost flat group has $(A,d)$-moderate growth.
\end{corollary}

The observation is that the almost flat condition implies the volume doubling condition at some relatively small scale (at most $(\diam_S(G))^{o(1)}$; see Lemma \ref{lem:5adic.ph}, below), and hence at all subsequent scales by Theorem  \ref{thm:scales}. Thus almost flat groups are doubling at all scales larger than $(\diam_S(G))^{o(1)}$. This is enough to guarantee moderate growth.
\bigskip

Groups with moderate growth provide examples of Cayley graphs for which the mixing time of the simple random walk behaves quadratically in the diameter. This is the main result proved by Diaconis and Saloff-Coste in
\cite{dsc}, which now translates as follows. We write $\mu_G$ for the uniform probability measure on $G$, $\mu_S$ for the uniform probability measure on $S$, and $\mu_S^{(n)}$ for the $n$-th convolution power of $\mu_S$.

\begin{corollary}[Mixing time]\label{cor:mixing}
The relaxation time and the mixing times of almost flat Cayley graphs are comparable to $(\diam_S(G))^2$. More precisely, given  $\eps>0$, there are positive constants $B,C,D>0$, depending on $\eps$ only, such that whenever $G$ is a finite group whose Cayley graph $(G,S)$ is $\eps$-almost flat and satisfies $\gamma  :=\diam_S(G)\ge D$, for all $n\geq 1$ and all  $p \in [1,\infty]$ we have
$$ \frac{1}{2}e^{-C \frac{n}{\gamma^2}} \|\mu_G\|_p \leq \|\mu_S^{(n)}-\mu_G\|_p\leq B e^{-\frac{n}{8|S|\gamma^2}} \|\mu_G\|_p,$$
where $\|\cdot\|_p$ denotes the $L^p$-norm on $G$. Indeed, the upper bound holds even without the assumption that $\gamma\ge D$.
\end{corollary}

We stress that the bounds for $C$ and $B$ obtained by Diaconis and Saloff-Coste are explicit when expressed in terms of the moderate growth parameters $A$ and $d$. In the above reformulation in terms of $\eps$-almost flatness, we lose the explicit dependence, because the moderate growth parameters $A$ and $d$ given by the second part of Corollary \ref{cor:diac.sc} depend on $\eps$ in a \emph{non-explicit} manner (due to our reliance on the non-explicit Theorem \ref{thm:bgt} below).

The mixing times (for each $L^p$ norm) are at least half the diameter, but for a general Cayley graph of bounded degree, they are always in $O(\gamma^3)$, where $\gamma$ is the diameter, as easily follows from the lower bound on $\lambda_1$ from $(\ref{groupbounds})$. This bound is attained for example for the uniform mixing time on the lamplighter group with cyclic base (e.g. see \cite{peres-revelle}). At the opposite extreme, if the Cayley graph is an expander graph (uniform lower bound on $\lambda_1$), then the mixing times are comparable to the diameter and a cut-off phenomenon is expected (and proved in some cases, see for that matter \cite{lubetzky-peres}). For almost flat Cayley graphs, no cut-off phenomenon arises.

We conclude with the following refinement of the Diaconis--Saloff-Coste result showing that doubling at one small enough scale is sufficient to guarantee quadratic mixing time.

\begin{theorem}[Quadratic mixing under a doubling condition]\label{mix} Given $K \geq 1$, there is $n_1(K)$ such that the following holds. Consider the Cayley graph $\mathcal{G}$ of a finite group $G$  with generating set $S$ and diameter $\gamma$. Assume that $|S^{2n+1}| \leq K|S^n|$ for some $n \geq n_1(K)$ with $n\leq \gamma^{2/3}$. Then the mixing times of the simple random walk on $\mathcal G$ are comparable to $\gamma^2$ up to multiplicative constants depending on $K$ and $|S|$ only.
\end{theorem}

We provide an example showing that the exponent $\frac{2}{3}$ is sharp; see Example \ref{sharpex}.

\bigskip

\bigskip

The paper is organised as follows. In Section 2, we discuss approximate groups and recall the main result of \cite{bgt}, which describes the structure of approximate subgroups of arbitrary groups. We also prove Theorem \ref{thm:scales} modulo a result on nilprogressions proved in Section 3. Section 3 is devoted to nilprogressions. We give a summary of several competing definitions and explain the relationship between them. Then we prove that their powers have uniform doubling, thus completing the proof of Theorem \ref{thm:scales}. Section 4 is devoted to Theorem \ref{thm:gromovfinite} and several extra structural results on almost flat groups. We also describe there an example pertaining to a non-commutative version of the Freiman-Ruzsa conjecture and showing that super-polynomial bounds are required in \cite{bgt}. In Section 5 we discuss bounds on the Cheeger constant and $\lambda_1$ for almost flat Cayley graphs and prove Corollaries \ref{cover-group} and \ref{cover-manifold} regarding property $(\tau)$. Finally Section 6 is devoted to random walks and mixing times. We establish Corollary \ref{cor:mixing} and Theorem \ref{mix}.

\subsection*{Acknowledgements}The authors are grateful to David Simmons, Romain Tessera and Tianyi Zheng for helpful conversations. We are also grateful to Persi Diaconis and Terry Tao for their comments on the paper. Both authors are supported by ERC grant GA617129 `GeTeMo'. M.T. is on leave from a Junior Research Fellowship at Homerton College, Cambridge, where some of this work was carried out.

\section{Approximate groups}\label{app-sec}

In recent years there has been much effort to describe the structure of certain finite subsets of groups called \emph{approximate subgroups}.  The reader may consult the surveys \cite{bgt.survey,app.grps,sand.survey} for details on the background to this field and its applications.  The definition of an approximate group is originally due to Tao \cite{tao.product.set}, and runs as follows.

\begin{definition}[Approximate group]\label{def:app.grp}
Let $G$ be a group and let $K\ge1$. A finite subset $A\subset G$ is said to be a \emph{$K$-approximate subgroup of $G$}, or just a \emph{$K$-approximate group}, if it is symmetric and if there is a subset $X$ with $|X|\le K$ such that $A^2\subset XA$.
\end{definition}

Here, and throughout the paper, for subsets $X$ and $Y$ of a group we denote by $XY$ the set $\{xy:x\in X,y\in Y\}$ and by $X^{-1}$ the set $\{x^{-1}:x\in X\}$.

It may not be immediately obvious that approximate groups should have anything to do with growth of groups. However, the following simple argument gives a reason to expect that the two should be related, at least on some level.
\begin{lemma}\label{lem:sm.doub.app.gp}
Let $G$ be a group, and let $S\subset G$ be a finite symmetric set. Suppose that
\begin{equation}\label{eq:sm.db.ap.gp}
|S^{5n}|\le K|S^n|,
\end{equation}
for some integer $n$, then $S^{2n}$ is a $K$-approximate subgroup of $G$.
\end{lemma}
\begin{proof} See Ruzsa's covering lemma \cite{tao-vu}. Upon replacing $S$ by $S^n$ we may assume that $n=1$. Let $X$ be a subset of $S^4$ that is maximal with respect to the condition that the sets $xS$ with $x\in X$ are disjoint, and note that this condition implies that $|XS|=|X||S|$. Given (\ref{eq:sm.db.ap.gp}), and since $XS\subset S^5$, this implies in particular that $|X|\le K$. It is therefore sufficient to show that $S^4\subset XS^2$. This holds, because if $z\in S^4$ then, by the maximality of $X$, there are $x\in X$ and $s_1,s_2\in S$ such that $zs_1=xs_2$. The symmetry of $S$ then implies that $z=xs_2s_1^{-1}\in XS^2$, as claimed.
\end{proof}

Recently, Green, Tao and the first author \cite{bgt} proved the following structure theorem for approximate subgroups of arbitrary groups.

\begin{theorem}[Structure of approximate groups {\cite[Theorem 1.6]{bgt}}]\label{thm:bgt}
Let $A$ be a $K$-approximate subgroup of a group $G$. Then there is a finite subgroup $H\subset A^4$ and a subgroup $\Gamma$ containing $H$ as a normal subgroup such that $\Gamma/H$ is nilpotent of rank and class at most $O_K(1)$, and such that $A$ is contained in fewer than $O_K(1)$ left-cosets of $\Gamma$.
\end{theorem}

As usual $O_K(1)$ denotes a constant that depends on $K$ only. Note that since the group $\Gamma$ is finite-by-nilpotent it is also virtually nilpotent. As remarked in \cite{bgt}, this result easily implies Gromov's original theorem on groups with polynomial growth. More precisely one has the following somewhat stronger statement.

\begin{corollary}[One-scale Gromov's theorem]\label{grom}
For every $K\ge1$ there exists $n_0=n_0(K)\in \N$ such that if $G$ is a group generated by a finite symmetric generating set $S$ such that $|S^{5n}|\le K|S^n|$ for some $n\ge n_0$ then $G$ is virtually nilpotent.
\end{corollary}

\begin{remark}Using Lemma \ref{doublingreduce} below, one sees that the hypothesis $|S^{5n}|\leq K|S^n|$ can be weakened into $|S^{2n+1}|\leq K|S^n|$. It is an interesting problem to determine whether one can push this further and assume only that $|S^{[\alpha n]}|\leq K|S^n|$ for some fixed $\alpha>1$.
\end{remark}

\begin{remark} If we were to allow the uniform doubling constant in the conclusion of Theorem \ref{thm:scales} to depend on the generating set $S$ then that theorem would become a straightforward application of Corollary \ref{grom}, as it is well known that large balls in Cayley graphs of finitely generated virtually nilpotent groups are doubling (see \cite{bass}, for example). However, the independence from $S$ is key to our applications.
\end{remark}

We recall the easily seen fact that if a group has polynomial growth, which is to say that $|S^n|=O(n^D)$ for some $D>0$, then there are infinitely many $n$ for which $|S^{5n}|< 6^D |S^n|$, and the hypotheses of Corollary \ref{grom} are satisfied at infinitely many scales. By contrast, the corollary asserts that only one large scale is necessary. For the reader's convenience, and also for further use below, we now recall how to derive this corollary from Theorem \ref{thm:bgt}, as was done in \cite[Corollary 11.2]{bgt}.

Theorem \ref{thm:bgt} shows that an approximate group lies in finitely many cosets of a nilpotent group. However, Corollary \ref{grom} requires that the \emph{entire} group lie in finitely many cosets of a nilpotent group. The key to overcoming this apparent difficulty is to observe that if the $n$-ball in a Cayley graph is contained in at most $n$ cosets of a subgroup, then this subgroup has index at most $n$, as follows.

\begin{lemma}\label{lem:exhausts.cosets}
Let $G$ be a group generated by a finite symmetric set $S$ containing the identity, and let $\Gamma$ be a subgroup of $G$. Suppose that for some $k\in\N$ we have $S^k$ contained in at most $k$ left-cosets of $\Gamma$. Then $G=S^k\Gamma$, and in particular $[G:\Gamma] \leq k$.
\end{lemma}
\begin{proof}
Since $S\Gamma\subset S^2\Gamma\subset S^3\Gamma\subset\cdots$, the number of left-cosets of $\Gamma$ having non-empty intersection with $S^r\Gamma$ is non-decreasing in $r$. Thus, if $S^k$ is contained in at most $k$ left-cosets of $\Gamma$, then $S^{r+1}\Gamma=S^r\Gamma$ for some $r\le k$.
It follows by induction that $G=S^r\Gamma$, and hence that $G=S^k\Gamma$.
\end{proof}

\begin{proof}[Proof of Corollary \ref{grom}] By Lemma \ref{lem:sm.doub.app.gp}, $S^{2n}$ is a $K$-approximate subgroup of $G$. By Theorem \ref{thm:bgt}, there exists $n_0=n_0(K)$ such that $S^{2n}$ is contained in at most $n_0$ cosets of a virtually nilpotent subgroup $\Gamma\le G$. If $n\ge n_0$ then Lemma \ref{lem:exhausts.cosets} therefore implies that $\Gamma$ is of finite index in $G$, and hence that $G$ is virtually nilpotent.
\end{proof}

We prove Theorem \ref{thm:scales} by first reducing it to the nilpotent case, which we state as follows.
\begin{theorem}[Nilpotent case]\label{prop:scales.nilp}
Let $K\ge1$ and $s\ge1$. Suppose that $S$ is $K$-approximate subgroup in some $s$-step nilpotent group $G$. Then there is constant $C=C(K,s)\geq 1$ such that for every $m\ge1$, $S^{m}$ is a $C$-approximate subgroup.
\end{theorem}
In the remainder of this section, we explain how to perform such a reduction using Theorem \ref{thm:bgt}. The proof of the nilpotent case will occupy the next section and rely on a study of nilprogressions.

\bigskip

\noindent \emph{Remark.} Contrary to the general case of Theorem \ref{thm:scales}, where no effective estimate is known on the size of $\theta(K)$, an effective bound in the nilpotent case can be easily computed by following our argument.

\bigskip

The reduction of Theorem \ref{thm:scales} to Theorem \ref{prop:scales.nilp} is based on the following proposition, which makes use of the structure theorem for approximate groups, Theorem \ref{thm:bgt}.

\begin{prop}\label{prop:v.nilp}
For every $K\ge1$ there exists $n_0=n_0(K)$ such that if $G$ is a group generated by a finite symmetric generating set $S$ such that
\begin{equation}\label{eq:cosets.doub}
|S^{5n}|\le K|S^n|
\end{equation}
for some $n\ge n_0$, then there are subgroups $H \leq \Gamma \leq G$ such that $H\subset S^{8n}$ is normal in $G$, such that $\Gamma/H$ is nilpotent of rank and class at most $O_K(1)$, and such that $\Gamma$ has index at most $n_0$ in $G$. Moreover, $B:=S^{8n}\cap\Gamma$  is an $O_K(1)$-approximate subgroup containing $H$, and there is a subset $X \subset S^n$ of size at most $n_0$ such that for every $b\in\N$ we have
\begin{equation}\label{eq:v.nilp}
S^{bn}\subset X B^{b-1}.
\end{equation}
\end{prop}

\begin{proof} By Lemma \ref{lem:sm.doub.app.gp}, $S^{2n}$ is a $K$-approximate subgroup of $G$.  We may therefore apply Theorem \ref{thm:bgt} to $S^{2n}$ and conclude that there is a natural number $n_1=n_1(K)$, a subgroup $H_0 \subset S^{8n}$, and a subgroup $\Gamma_0 \leq G$ containing $H_0$ as a normal subgroup, such that $\Gamma_0/H_0$ is a nilpotent group of rank and step at most $O_K(1)$, and such that $S^{2n}$ is contained in at most $n_1$ cosets of $\Gamma_0$. If $n \geq n_1$ then Lemma \ref{lem:exhausts.cosets} implies that $\Gamma_0$ has index at most $n_1$ in $G$. In particular, $\Gamma_0$ has a normal subgroup $\Gamma$ of index at most $n_0:=n_1!$.

Since $\Gamma$ is normal in $G$, so is $C^s(\Gamma)$, the $s$-th term in its descending central series, where $s$ is the nilpotency class of $\Gamma_0/H_0$. Setting $H:=C^s(\Gamma)$, we see that $\Gamma/H$ is $s$-step nilpotent and that $H \subset H_0 \subset S^{8n}$ is normal in $G$. It is well known that if $F$ is an approximate subgroup then so is the intersection of $F^2$ with any subgroup (see \cite[Lemma 2.10]{tointon}, for example). This implies in particular that $B:=S^{8n} \cap \Gamma$ is an $O_K(1)$-approximate subgroup.

It remains to exhibit a set $X\subset S^n$ satisfying $(\ref{eq:v.nilp})$. If $n \geq n_0$ then $S^n$ intersects every coset of $\Gamma$ in $G$, and so $S^n$ contains a complete set of coset representatives of $\Gamma$; we claim that it suffices to take $X$ to be this set. Indeed, by definition this set satisfies $S^{2n} \subset X\Gamma$, and hence $S^{2n} \subset X(S^{8n} \cap \Gamma)= XB$. In particular $S^{2n} \subset S^nB$, and so $S^{bn} \subset S^{2n} B^{b-2}$ for each $b \in \N$ by induction. The result follows.
\end{proof}

The following combinatorial lemma allows us to pass from a small doubling assumption to the more comfortable $|S^{5n}|\leq K|S^n|$.

\begin{lemma}\label{doublingreduce} For every $K\geq 2$, there is $C(K)\geq 2$ such that if $S$ is a finite symmetric subset (containing $1$) in a group $G$, and if $n \geq 1$ is an  integer such that
$$|S^{2n+1}| \leq K|S^n|,$$ then $|S^{5n}| \leq C(K)|S^n|$.
\end{lemma}

\begin{proof} We apply Petridis's lemma \cite{petridis} (see also \cite{ruzsa.plu}), according to which if $A$ and $B$ are finite subsets of a group $G$ such that $|BA| \leq K|A|$, then there is a subset $A_0 \subset A$ that $|BA_0X| \leq K |A_0X|$ for every finite subset $X$ in $G$. Applying this lemma with $B=S^{n+1}$ and $A=S^n$ gives a subset $A_0$ of $S^n$ with the property that $$|S^{n+1}A_0X| \leq K|A_0X|$$ for all $X \subset G$. In particular, this implies that $|S^{n+1}| \leq |S^{n+1}A_0| \leq K|A_0|$, and hence that $|S^{2n+1}| \leq K|S^n| \leq K^2 |A_0|$. This also shows that $|A_0SA_0| \leq K|A_0|$, since $A_0S \subset S^{n+1}$. Finally, taking $X=A_0$ and noting that $A_0 \subset S^{n} \subset S^{n+1}$, we see that $|A_0^3| \leq K|A_0^2|\leq K^2 |A_0|$.

Now set $A_1:=S \cup A_0$. Note that $A_1^3$ is contained in the union of the subsets $A_0^3$, $S^3$, $A_0SA_0$, $S^2A_0$, $A_0S^2$, $SA_0S$, $A_0^2S$, $SA_0^2$, of which all but $A_0^3$ are contained in $S^{2n+1}$. Since both $A_0^3$ and $S^{2n+1}$ are of size at most $K^2 |A_0|$, this implies that $|A_1^3| \leq 8K^2|A_1|$, and so $A_1$ has small tripling. It is well known that small tripling implies small $k$-pling for every integer $k$; more precisely, \cite[Lemma 3.4]{tao.product.set}, for example, implies that $|(A_1 \cup A_1^{-1})^k| \leq O_{K,k}(1)|A_1|$.

We claim that there exists $k=O_K(1)$ such that $S^n$ is contained in $(A_1 \cup A_1^{-1})^k$. This will finish the proof of the lemma, since it implies that $S^{5n} \subset (A_1 \cup A_1^{-1})^{5k}$, and hence that $|S^{5n}| \leq O_{K,k}(1) |A_1| \leq O_K(1)|S^n|$.

We build recursively sequences $x_i \in S^n$ and $k_i\in\Z$ for $i \geq 0$. Take $x_0$ to be the identity, and $k_0$ to be zero. Then, assuming $x_0,...,x_i$ and $k_1,\ldots,k_i$ have been chosen, let $k_{i+1}$ be the smallest integer such that there is $y \in S^{k_{i+1}}$ with the property that $yA_1$ is disjoint from all $x_jA_1$ for $j=0,\ldots,i$, and let $x_{i+1}$ be such an element $y$.

Let $m$ be the largest integer such that $k_m \leq n$ but $k_{m+1} > n$. Since the $x_iA_1$ are all disjoint and belong to $S^nA_1 \subset S^{2n}$, there can be at most $|S^{2n}|/|A_1| \leq K^2$ of them, and so $m \leq K^2$.

On the other hand, we claim that $x_i \in (A_1 \cup A_1^{-1})^{3i}$. We verify this by induction. Assuming this holds for $i$, we write $x_{i+1}=su$, where $u \in S^{k_{i+1}-1}$ and $s \in S$. By construction, $uA_1$ must intersect some $x_jA_1$ non-trivially for some $j \leq i$. This means that $u \in x_jA_1A_1^{-1} \subset (A_1 \cup A_1^{-1})^{3i+2}$. However, $S \subset A_1$, and hence $x_{i+1} \in  (A_1 \cup A_1^{-1})^{3i+3}$, as desired.

If $x \in S^n$ is arbitrary then, by the definition of $x_m$, there must be some $j\leq m$ such that $xA_1$ intersects $x_jA_1$ non-trivially. This means that $x \in x_jA_1A_1^{-1}$, and hence that $x \in (A_1 \cup A_1^{-1})^{3m+2}$.  We conclude that $S^n \subset (A_1 \cup A_1^{-1})^{3m+2}$. The claim, and hence the lemma, now follows, since $m\le K^2$.
\end{proof}

%
%

\begin{lemma}\label{lem:pullback}
Let $\pi:G\to G'$ be a homomorphism with kernel $N$, and suppose that $A$ is a symmetric subset of $G$ such that $N\subset A$ and such that $\pi(A)$ is a $K$-approximate group. Then $A^2$ is a $K^3$-approximate group.
\end{lemma}
\begin{proof}
To say that $\pi(A)$ is a $K$-approximate group is to say that there exists $X\subset G$ with $|X|\le K$ such that $A^2\subset XAN$. Iterating this, we conclude that $A^4\subset X^3AN\subset X^3A^2$.
\end{proof}

We can now complete the reduction of Theorem \ref{thm:scales} to the nilpotent case.
\begin{proof}[Proof of Theorem \ref{thm:scales} modulo Theorem \ref{prop:scales.nilp}] Applying Lemma \ref{doublingreduce}, we may assume that $|S^{5n}|\leq K|S^n|$. Proposition \ref{prop:v.nilp} therefore gives a finite-index subgroup $\Gamma$ of $G$ and a normal subgroup $H$ contained in $S^{8n}$ such that $\Gamma/H$ is nilpotent with bounded rank and nilpotency class, such that  $S^{bn} \subset XB^b$ for every $b\geq 1$, where $|X|\leq n_0(K)$ and $B:=S^{8n}\cap \Gamma$ is an $O_K(1)$-approximate subgroup of $\Gamma$ containing $H$. Applying Theorem \ref{prop:scales.nilp} to the nilpotent group $\Gamma/H$, and pulling back to $\Gamma$ using Lemma \ref{lem:pullback}, we conclude that for all $b\in \N$ the set $B^{2b}$ is an $O_K(1)$-approximate group. In particular, $S^{32bn} \subset XB^{32b} \subset XY_b B^{2b}$, for some subset $Y_b$ of size at most $O_K(1)$. Since $B^{2b} \subset S^{16bn}$, we conclude that $S^{16bn}$ is an $O_K(1)$-approximate subgroup, uniformly for all $b \in \N$. Write $X_{16bn}$ for the set $XY_b$ witnessing this.

Given $m\ge16n$, let $m'$ be the smallest integer multiple of $32n$ that is at least $m$, and note that $m\ge m'/2$, and hence that $|S^{m'}|\le O_K(1)|S^{m'/2}|\le O_K(1)|S^m|$. We also have
\[
S^{cm}\subset S^{cm'}\subset X_{m'}^{c-1}S^{m'},
\]
from which the theorem then follows easily for the value of $m$ in question. For $m<16n$, on the other hand, the theorem follows from Lemma \ref{lem:sm.doub.app.gp}.
\end{proof}
%
%

%
%

\section{Nilprogressions}\label{sec:coms.nilprogs}
Our goal in this section is to prove Theorem \ref{prop:scales.nilp}. This will conclude the proof of Theorem \ref{thm:scales}. The proof relies on the structure theorem for nilpotent approximate groups due to the second author \cite{tointon}, which we recall below as Theorem \ref{thm:nilp}, as well as an analysis of nilprogressions and their powers, which forms the basis of Proposition \ref{powers}, below. 

Nilprogressions are to nilpotent groups what arithmetic progressions are to the infinite cyclic group $\Z$. There are several natural  ways to define this concept. Although they do not lead to exactly the same notions, they are roughly equivalent on some level, as we will see. We start with the following definition.

\begin{definition}[Nilprogression, see \cite{bgt}]\label{def:nilprog}
Let $G$ be a group, let $x_1,\ldots,x_r\in G$ and let $L=(L_1,\ldots,L_r)$ be a vector of non-negative integers. If the $x_1,\ldots,x_r$ generate an $s$-step nilpotent subgroup of $G$ then the set of all products in the $x_i$ and their inverses, in which each $x_i$ and its inverse appear at most $L_i$ times between them, is said to be a \emph{nilprogression} of rank $r$ and step $s$, and is denoted $P^*(x_1,\ldots,x_r;L)$.
\end{definition}

For the purposes of this paper we need to introduce a new variant of a nilprogression, which we call a \emph{nilcomplete progression}. This is closely related to yet another variant, introduced by Green and the first author in \cite{tor.free.nilp}, where it is called a \emph{nilpotent progression}.  We show below precisely how they may be thought of as `roughly equivalent'; see \eqref{eq:nested.progs} in particular. For further background on this rough equivalence the reader may consult the work \cite{tointon} of the second author.

In order to define nilcomplete progressions, we need to establish some terminology. The following definition is reproduced from \cite{tointon}, and follows a set up in \cite[\S1]{tor.free.nilp} that was in turn based on \cite[\S11.1]{hall}.

\begin{definition}[Commutators and weights]
We define \emph{(formal) commutators} in the letters $x_1,\ldots,x_r$ recursively by defining each $x_i$ to be a formal commutator and, given two formal commutators $\alpha,\alpha'$ in the $x_j$, defining $[\alpha,\alpha']$ also to be a formal commutator.

To each commutator $\alpha$ we assign a \emph{weight vector} $\chi(\alpha)\in\N_0^r$, defined recursively by setting $\chi(x_i):=e_i$ and, given two formal commutators $\alpha,\alpha'$ in the $x_j$, defining $\chi([\alpha,\alpha'])=\chi(\alpha)+\chi(\alpha')$. We define the \emph{total weight} $|\chi(\alpha)|$ of a commutator $\alpha$ to be $\|\chi(\alpha)\|_1$. Given a weight vector $\chi\in\N_0^r$ and a vector $L=(L_1,\ldots,L_r)$ of positive integers we write $L^\chi$ to denote the quantity 
$$L^\chi:=L_1^{\chi_1}\cdots L_r^{\chi_r}.$$

Noting that this results in at most finitely many commutators of any given weight vector, we assign a fixed total ordering $\prec$ to the commutators, chosen arbitrarily subject to the conditions that $x_1\prec\ldots\prec x_r$, that commutators of the same weight vector are consecutive, and that commutators of lower total weight come before commutators of higher total weight.

Finally, for each commutator $\alpha$ we define the \emph{(formal) inverse commutator} $\alpha^{-1}$.
We extend $\prec$ to a partial ordering of commutators and their formal inverses, defining $\alpha^{\pm1}\prec\beta^{\pm1}$ when $\alpha\prec\beta$.
\end{definition}

Certain commutators in the $x_1,\ldots,x_r$ are called \emph{basic commutators}. These are defined precisely in each of \cite{tor.free.nilp,hall,tointon}; for the purposes of this paper it will suffice to note that if $G$ is the free $s$-step nilpotent group on generators $x_1,\ldots,x_r$, and $c_1,\ldots,c_{t'}$ is the ordered (with respect to $\prec$) list of basic commutators of weight at most $s$ in the $x_i$, then the elements
\[
c_1^{l_1}\cdots c_{t'}^{l_{t'}}
\]
are all distinct as the $l_i$ range over $\Z$ and, moreover, every element of $G$ can be expressed in this form.

\begin{definition}[Nilpotent progression, see \cite{tor.free.nilp}]
Let $G$ be a group, let $x_1,\ldots,x_r$ be elements of $G$ and let $L=(L_1,\ldots,L_r)$ be a vector of non-negative integers. If $x_1,\ldots,x_r$ generate an $s$-step nilpotent subgroup of $G$ then, writing $c_1,\ldots,c_{t'}$ for the ordered (with respect to $\prec$) list of basic commutators of weight at most $s$ in the $x_i$, the set
\[
P(x_1,\ldots,x_r;L):=\{c_1^{l_1}\cdots c_{t'}^{l_{t'}}:|l_i|\le L^{\chi(c_i)}\}
\]
is said to be a \emph{nilpotent progression} of rank $r$ and step $s$. If the elements $c_1^{l_1}\cdots c_{t'}^{l_{t'}}$ are all distinct as the $l_i$ range over $[-L^{\chi(c_i)},L^{\chi(c_i)}]$ then the nilpotent progression $P(x_1,\ldots,x_r;L)$ is said to be \emph{proper}.
\end{definition}

\begin{remark}The nilpotent progression $P(x_1,\ldots,x_r;L)$ is proper, for example, in the free $s$-step nilpotent group generated by $x_1,\ldots,x_r$.
\end{remark}
\begin{remark}\label{rem:proper}
The cardinality of a nilpotent progression $P(x_1,\ldots,x_r;L)$ is easily seen to be at most $\prod_{i=1}^t(2L^{\chi(c_i)}+1)$, with equality if and only if $P(x_1,\ldots,x_r;L)$ is proper. In the case that $P(x_1,\ldots,x_r;L)$ is proper, this implies in particular that $|P(x_1,\ldots,x_r;ML)|\le M^{O_{r,s}(1)}|P(x_1,\ldots,x_r;L)|$ for every $M\in\N$.
\end{remark}
We now move on to the definition of a nilcomplete progression.
The reason for introducing nilcomplete progressions is that they behave particularly well on taking powers or on scaling the side lengths $L_i$ (see Proposition \ref{powers} below).
The definition is similar to that of a nilpotent progression, except that we use the list of \emph{all} commutators in the $x_i$ instead of just the basic commutators, and that when taking commutators we allow lower-weight components to be replaced by their inverses (so, for example, we allow commutators such as $[x_1,x_2^{-1}]$ and $[x_3,[x_1^{-1},x_2]^{-1}]$). We express this more formally as follows.
\begin{definition}[Generalised commutators]
We define the \emph{generalised commutators} in the $x_i$ in exactly the same way as the commutators, except that if $\alpha_1$ and $\alpha_2$ are two generalised commutators in the $x_i$ then so is $[\alpha_1^{\epsilon_1},\alpha_2^{\epsilon_2}]$ for every $\epsilon_1=\pm1$ and $\epsilon_2=\pm1$.

We extend the definition of weight vectors from commutators to generalised commutators by requiring that the weight vectors of $[\alpha_1^{\epsilon_1},\alpha_2^{\epsilon_2}]$ are equal for all choices of $\epsilon_1$ and $\epsilon_2$. We extend the order $\prec$ arbitrarily so that the generalised commutators $[\alpha_1^{\epsilon_1},\alpha_2^{\epsilon_2}]$ with different choices of $\epsilon_1$ and $\epsilon_2$ are consecutive.
\end{definition}
Throughout this section we write $\xi_1,\ldots,\xi_t$ for the ordered (with respect to $\prec$) list of generalised commutators of total weight at most $s$ in $x_1,\ldots,x_r$. Note that $t$ is bounded in terms of $r$ and $s$ only.

\begin{definition}[Nilcomplete progression]
Let $G$ be a group, let $x_1,\ldots,x_r$ be elements of $G$, and let $L=(L_1,\ldots,L_r)$ be a vector of non-negative integers. If $x_1,\ldots,x_r$ generate an $s$-step nilpotent subgroup of $G$ then, writing $\xi_1,\ldots,\xi_t$ for the ordered (with respect to $\prec$) list of generalised commutators of weight at most $s$ in the $x_i$, the set
\[
\overline P(x_1,\ldots,x_r;L)=\{\xi_1^{l_1}\xi_2^{l_2}\cdots \xi_t^{l_t}:|l_i|\le L^{\chi(\xi_i)}\}
\]
is said to be a \emph{nilcomplete progression} of rank $r$ and step $s$.
\end{definition}

We recall a technical notion that was useful in \cite{tointon}.
\begin{definition}[Ordered progression]Let $x_1,\ldots,x_r$ be elements of a group and let $L=(L_1,\ldots,L_r)$ be a vector of non-negative integers. Then the set
\[
P_\textup{ord}(x_1,\ldots,x_r;L):=\{x_1^{l_1}\cdots x_r^{l_r}:|l_i|\le L_i\}
\]
is said to be an \emph{ordered progression} of rank $r$.
\end{definition}

It is immediate from the definitions that
\begin{equation}\label{eq:closed.as.ord}
\overline P(x_1,\ldots,x_r;L)=P_\textup{ord}(\xi_1,\ldots,\xi_t;L^{\chi(\xi_1)},\ldots,L^{\chi(\xi_t)}).
\end{equation}

It was shown in \cite{tointon} that nilprogressions, nilpotent progressions and nilcomplete progressions are `essentially' equivalent, and that nilcomplete progressions control nilpotent approximate groups. In particular, we have
\begin{equation}\label{eq:nested.progs}
P_\textup{ord}(x_1,\ldots,x_r;L)\subset P^*(x_1,\ldots,x_r;L)\subset P(x_1,\ldots,x_r;L)\subset\overline P(x_1,\ldots,x_r;L)\subset P_\textup{ord}(x_1,\ldots,x_r;L)^{r^{O_s(1)}}.
\end{equation}
Here the first and third inclusions are trivial, while the second follows from the commutator collecting process, as described in \cite{tointon}. The last inclusion is proven in \cite[Proposition C.1]{tointon} for $P$ in place of $\overline P$, but precisely the same argument works and yields \eqref{eq:nested.progs}.

\bigskip

Next we recall the main result of \cite{tointon}.

\begin{theorem}[Structure of nilpotent approximate groups, \cite{tointon}]\label{thm:nilp}
Let $G$ be an $s$-step nilpotent group, and suppose that $A$ is a $K$-approximate subgroup of $G$. Then there is a subgroup $H$ normalised by $A$, a natural number $r\le K^{O_s(1)}$, elements $x_1,\ldots,x_r\in G$, and a vector $L=(L_1,\ldots,L_r)$ of non-negative integers such that
\[
A\subset HP^*(x_1,\ldots,x_r;L)\subset HP(x_1,\ldots,x_r;L)\subset H\overline P(x_1,\ldots,x_r;L)\subset A^{K^{O_s(1)}}.
\]
\end{theorem}
Again, this statement was proven in \cite{tointon} only for $P$, and not for $\overline{P}$, but precisely the same proof yields the result for $\overline{P}$ as well.

We are now ready to state the main technical result of this section, which is a description of the properties of nilcomplete progressions and their powers.

\begin{prop}[Powers of nilcomplete progressions]\label{powers}Let $G$ be an $s$-step nilpotent group, let $x_1,\ldots,x_r$ be elements of $G$, and let $L=(L_1,\ldots,L_r)$ be a vector of positive integers. Then the following conditions hold.
\begin{enumerate}
\item \label{prop:dilates.1} $\overline P(x_1,\ldots,x_r;nL)\subset\overline P(x_1,\ldots,x_r;L)^{O_{r,s}(n)}$.
\item \label{prop:dilates.2} $\overline P(x_1,\ldots,x_r;L)^n\subset\overline P(x_1,\ldots,x_r;nL)$.
\item \label{prop:translates} For every $M\in\N$ there is some set $X\subset G$ of cardinality at most $O_{r,s,M}(1)$ such that
\[
\overline P(x_1,\ldots,x_r;ML)\subset\overline P(x_1,\ldots,x_r;L)X.
\]
\end{enumerate}

\end{prop}

\begin{remark}
An inspection of the arguments of \cite{tointon} reveals that the exponent $r^{O_s(1)}$ in (\ref{eq:nested.progs}) could be written more precisely as $e^{O(s^4)}r^{O(s)}$, or simply $O_s(r^{O(s)})$. An inspection of the arguments below then shows that the constant implied by $O_{r,s}(n)$ in Proposition \ref{powers} (\ref{prop:dilates.1}) could be written as $e^{O(s^4)}r^{O(s)}n$, or more succinctly as $O_s(r^{O(s)}n)$. Similarly, the $O_{r,s,M}(1)$ in Proposition \ref{powers} (\ref{prop:translates}) could be written as $r^{O(s^3r^{O(s)})}M^{sr^{O(s)}}$, or more succinctly as $O_{r,s}(M^{sr^{O(s)}})$. We leave the details to the interested reader.
\end{remark}

\bigskip

It is convenient at this point to introduce some further notation. Given a list $y_1,\ldots,y_k$ of commutators in the $x_i$ we denote by $L^{\chi(y)}$ the $k$-dimensional vector
\[
L^{\chi(y)}:=(L^{\chi(y_1)},\ldots,L^{\chi(y_k)}).
\]
Thus, for example, (\ref{eq:closed.as.ord}) can be rewritten more succinctly as
\begin{equation}\label{eq:closed.as.ord'}
\overline P(x_1,\ldots,x_r;L)=P_\textup{ord}(\xi_1,\ldots,\xi_t;L^{\chi(\xi)}).
\end{equation}
In the event that $L=(L_1,\ldots,L_r)$ is a vector of real numbers, as opposed to integers, we define $\lfloor L\rfloor$ and $\lceil L\rceil$ coordinate-wise; thus, for example,
$\lfloor L\rfloor=(\lfloor L_1\rfloor,\ldots,\lfloor L_r\rfloor)$. When $L'$ is another vector of integers or real numbers, we write $L<L'$ or $L\le L'$ to indicate that the appropriate inequality holds coordinate-wise.

Finally we introduce one further definition.

\begin{definition}[Nilcomplete set of generalised commutators]
A set $Y$ of generalised commutators in a list of group elements $x_1,\ldots,x_r$ is said to be \emph{nilcomplete} if whenever $\alpha_1,\alpha_2\in Y$ and $\epsilon_1=\pm1$ and $\epsilon_2=\pm1$ are such that $[\alpha_1^{\epsilon_1},\alpha_2^{\epsilon_2}]\ne1$, the generalised commutator $[\alpha_1^{\epsilon_1},\alpha_2^{\epsilon_2}]$ also belongs to $Y$.
\end{definition}

We now pass to the proof of Proposition \ref{powers}.

\begin{proof}[Proof of Proposition \ref{powers} (\ref{prop:dilates.1})] We set $B(L):=\bigcup_1^r \{x_i^{\ell_i} : |\ell_i| \leq L_i\}$,
 and use the shorthand $\overline{P}(L):=\overline{P}(x_1,\ldots,x_r;L)$. By definition of $\overline{P}$ (see (\ref{eq:closed.as.ord'})) we have $\overline{P}(nL)\subset P(\xi_1;(nL)^{\chi(\xi_1)})\cdots P(\xi_t;(nL)^{\chi(\xi_t)})$. However, \cite[Lemma C.2]{tointon} implies that
\[
P(\xi_i;(nL)^{\chi(\xi_i)})\subset B(nL)^{O_s(1)}.
\]
Noting again that $t$ is bounded in terms of $r$ and $s$ only, we see that
\begin{equation}\label{eq:PinB}
\overline{P}(nL)\subset B(nL)^{O_{r,s}(1)}\subset B(L)^{O_{r,s}(n)},
\end{equation}
and the result then follows from the fact that $B(L) \subset \overline{P}(L)$.
\end{proof}

\begin{proof}[Proof of Proposition \ref{powers} (\ref{prop:dilates.2})]
We in fact show that
\[
P_\textup{ord}(\xi_1,\ldots,\xi_t;L^{\chi(\xi)})^n\subset P_\textup{ord}(\xi_1,\ldots,\xi_t;(nL)^{\chi(\xi)});
\]
in light of (\ref{eq:closed.as.ord'}) this is sufficient to prove Proposition \ref{powers} (\ref{prop:dilates.2}). 

First, note the trivial identity
\begin{equation}\label{eq:collecting.identity.1}
vu=uv[v,u].
\end{equation}
Since the $\xi_i$ are generalised commutators, it is very straightforward to use this identity repeatedly to see that any string $\sigma$ in the $x_i$ is expressible in the form
\begin{equation}\label{eq:collected}
\xi_1^{l_1}\cdots\xi_t^{l_t}
\end{equation}
for \emph{some} $l_1,\ldots,l_t\in\Z$. For example, the string $x_2x_1^{-1}$ is equal to $x_1^{-1}x_2[x_2,x_1^{-1}]$. This is essentially a simplified version of the \emph{collecting process} of \cite[\S11.1]{hall}. The challenge is to show that if $\sigma$ is of the form $p^{(1)}\cdots p^{(n)}$ with each $p^{(i)}$ belonging to $P_\textup{ord}(\xi_1,\ldots,\xi_t;L^{\chi(\xi)})$ then we may conclude additionally that $|l_i|\le(nL)^{\chi(\xi_i)}$.

Let $P^{(1)},\ldots,P^{(n)}$ be distinct copies of $P_\textup{ord}(\xi_1,\ldots,\xi_t;L^{\chi(\xi)})$, so that $P_\textup{ord}(\xi_1,\ldots,\xi_t;L^{\chi(\xi)})^n$ may be written as $P^{(1)}\cdots P^{(n)}$. We bound from above the number of instances of each generalised commutator $\xi_i$ that can occur when for each $j$ we have an element $p^{(j)}$ of $P^{(j)}$ viewed as a string in the $\xi_i$ in the natural way, and the simplified collecting process is applied to the string $p^{(1)}\cdots p^{(n)}$.

In order to do this, we give a different label to each occurrence of each generalised commutator $\xi_i$ in the initial string $p^{(1)}\cdots p^{(n)}$. Specifically, given the string $p^{(j)}=\xi_1^{l_1^{(j)}}\cdots\xi_t^{l_t^{(j)}}$ in $P^{(j)}$, for each $i=1,\ldots,t$ and for $i'=1,\ldots,l_i$ let $\xi_{i,i'}^{(j)}$ be distinct copies of the element $\xi_i$. Thus we have
\[
\xi_1^{l_1^{(j)}}\cdots\xi_t^{l_t^{(j)}}=\left(\xi_{1,1}^{(j)}\cdots\xi_{1,l_1^{(j)}}^{(j)}\right)\cdots\left(\xi_{t,1}^{(j)}\cdots\xi_{t,l_t^{(j)}}^{(j)}\right)
\]
and, in particular,
\begin{equation}\label{eq:long.string}
p^{(1)}\cdots p^{(n)}=\left(\left(\xi_{1,1}^{(1)}\cdots\xi_{1,l_1^{(1)}}^{(1)}\right)\cdots\left(\xi_{t,1}^{(1)}\cdots\xi_{t,l_t^{(1)}}^{(1)}\right)\right)\cdots\left(\left(\xi_{1,1}^{(n)}\cdots\xi_{1,l_1^{(n)}}^{(n)}\right)\cdots\left(\xi_{t,1}^{(n)}\cdots\xi_{t,l_t^{(n)}}^{(n)}\right)\right).
\end{equation}
Let us apply (\ref{eq:collecting.identity.1}) repeatedly to this string so that when we once again identify each occurrence of the same generalised commutator it is of the form (\ref{eq:collected}). Specifically, beginning with the left-most occurrence of $\xi_1$, use identity (\ref{eq:collecting.identity.1}) to move each instance of $\xi_1$ to the left until it is to the left of every instance of every commutator that is later with respect to the order $\prec$; then perform the same process with the instances of $\xi_2$; and so on up to the instances of $\xi_t$ (which by this time, of course, will be at the extreme right of the expression anyway).

Once this is done, the resultant string of the form (\ref{eq:collected}) will consist of commutators whose arguments are elements of the form $\xi_{i,i'}^{(j)}$. Some of these will simply be the original elements $\xi_{i,i'}^{(j)}$ from the string (\ref{eq:long.string}); let us call such elements \emph{prime} instances of the commutator $\xi_i$. However, suppose the commutator $\xi_i$ is equal to $[\xi_v,\xi_u]$, with $u<v$. Then there may also be instances of $\xi_i$ in (\ref{eq:collected}) that resulted from interchanging instances of $\xi_u$ with instances of $\xi_v$; let us call such instances of $\xi_i$ \emph{composite} instances. Note that a composite instance of the commutator $[\xi_v,\xi_u]$ could have arisen from interchanging a composite instance of $\xi_u$ with a composite instance of $\xi_v$.

Write $p_i^{(j)}$ for the number of prime instances of $\xi_i$ originating from the progression $P^{(j)}$, and $c_i$ for the total number of composite instances of $\xi_i$. Our aim is to prove that $p_i^{(1)}+\ldots+p_i^{(n)}+c_i\le(nL)^{\chi(\xi_i)}$. Since
\begin{equation}\label{eq:prime.from.j}
p_i^{(j)}\le L^{\chi(\xi_i)}
\end{equation}
for each $j$ by definition, this amounts to showing that
\begin{equation}\label{eq:dilates.concl}
c_i\le(nL)^{\chi(\xi_i)}-nL^{\chi(\xi_i)}.
\end{equation}
This is trivial for $i=1,\ldots,r$, since for each value of $i$ less than $r$ we have $c_i=0$, and so by induction we may prove (\ref{eq:dilates.concl}) for an arbitrary fixed value $w$ of $i$, under the assumptions that $w>r$ and that (\ref{eq:dilates.concl}) holds for all smaller values of $i$. Since $w>r$ we have $\xi_w=[\xi_v,\xi_u]$ for some $u$ and $v$ satisfying $u<v<w$.

Let us first bound from above the number of composite instances of $\xi_w$ in which the $\xi_u$ component is prime. Indeed, consider first the number of composite instances of $\xi_w$ in which the $\xi_u$ component is a prime instance originating from $P^{(j)}$. Prime instances of commutators originating in the same progression $P^{(j)}$ are in the desired order from the beginning, and so will never swap under the simplified collecting process, and so if $\zeta$ is a composite instance of $\xi_w$ in which the $\xi_u$ component is a prime instance originating from $P^{(j)}$, then its $\xi_v$ component cannot be a prime instance originating from $P^{(j)}$.

The number of instances of $\xi_v$ that are not prime instances coming from $P^{(j)}$ is at most $(nL)^{\chi(\xi_v)}-L^{\chi(\xi_v)}$ by (\ref{eq:prime.from.j}) and the induction hypothesis. On the other hand, (\ref{eq:prime.from.j}) implies that the number of prime instances of $\xi_u$ originating from $P^{(j)}$ is at most $L^{\chi(\xi_u)}$. The total number of composite instances of $\xi_w$ in which the $\xi_u$ component is a prime instance originating from $P^{(j)}$ is therefore at most
\[
(n^{|\chi(\xi_v)|}-1)L^{\chi(\xi_w)},
\]
and so the total number of composite instances of $\xi_w$ in which the $\xi_u$ component is a prime instance originating from \emph{any} progression is at most
\begin{equation}\label{eq:total.prime}
(n^{|\chi(\xi_v)|+1}-n)L^{\chi(\xi_w)}.
\end{equation}
Let us now bound the number of composite instances of $\xi_w$ in which the $\xi_u$ component is composite. By induction, the number of composite instances of $\xi_u$ is at most $(n^{|\chi(\xi_u)|}-n)L^{\chi(\xi_u)}$, whilst by induction and (\ref{eq:prime.from.j}) the total number of instances of $\xi_v$ is at most $n^{|\chi(\xi_v)|}L^{\chi(\xi_v)}$. The number of composite instances of $\xi_w$ in which the $\xi_u$ component is composite is therefore at most
\begin{equation}\label{eq:total.comp}
(n^{|\chi(\xi_w)|}-n^{|\chi(\xi_v)|+1})L^{\chi(\xi_w)}.
\end{equation}
The total number of composite instances of $\xi_w$ is therefore at most the sum of the quantities (\ref{eq:total.prime}) and (\ref{eq:total.comp}), which is $(n^{|\chi(\xi_w)|}-n)L^{\chi(\xi_w)}$. This verifies (\ref{eq:dilates.concl}) in the case $i=w$, and so ends the proof.
\end{proof}

In light of (\ref{eq:closed.as.ord'}), Proposition \ref{powers} (\ref{prop:translates}) follows from the following more general statement.
\begin{prop}\label{prop:translates'}
Let $y_1,\ldots,y_k$ be a nilcomplete set of generalised commutators in the $x_i$, appearing in order with respect to $\prec$. Then for every $M\in\N$ there is some set $X\subset G$ of size bounded in terms of $k,s,M$ such that
\[
P_\textup{ord}(y_1,\ldots,y_k;(ML)^{\chi(y)})\subset P_\textup{ord}(y_1,\ldots,y_k;L^{\chi(y)})X.
\]
\end{prop}

The first containment of (\ref{eq:nested.progs}) expresses the trivial fact that an ordered progression is contained in inside the corresponding nilprogression. The following useful lemma shows that the reverse containment is also approximately true if the generators form a nilcomplete set.

\begin{lemma}\label{lem:split.prog}
Let $y_1,\ldots,y_k$ be a nilcomplete set of generalised commutators in the $x_i$, appearing in order with respect to $\prec$. Then
\[
P^*(y_1,\ldots,y_k;L^{\chi(y)})\subset P_\textup{ord}(y_1,\ldots,y_k;(O_{k,s}(1)L)^{\chi(y)}).
\]
\end{lemma}
\begin{proof}
Since the $y_i$ are in order with respect to $\prec$, for each $j$ the sequence $y_j,\ldots,y_k$ is still nilcomplete, and so by induction on $k$ it suffices to prove that
\[
P^*(y_1,\ldots,y_k;L^{\chi(y)})\subset P(y_1;L^{\chi(y_1)})P^*(y_2,\ldots,y_k;(O_{k,s}(1)L)^{\chi(y_2)},\ldots,(O_{k,s}(1)L)^{\chi(y_k)}).
\]
By (\ref{eq:nested.progs}) it is in fact sufficient to prove that
\[
P(y_1,\ldots,y_k;L^{\chi(y)})\subset P(y_1;L^{\chi(y_1)})P^*(y_2,\ldots,y_k;(O_{k,s}(1)L)^{\chi(y_2)},\ldots,(O_{k,s}(1)L)^{\chi(y_k)}).
\]
However, the fact that the set of $y_i$ is nilcomplete implies that each basic commutator in the $y_i$ is itself equal to some $y_j$, and so this inclusion follows from the fact that there are at most $O_{k,s}(1)$ basic commutators in the $y_i$ of total weight at most $s$.
\end{proof}
\begin{lemma}\label{lem:comm.p*}
Let $z,y_1,\ldots,y_k$ be a nilcomplete set of generalised commutators in the $x_i$, appearing in order with respect to $\prec$. Let $m\in\Z$, and suppose that
\begin{equation}\label{eq:m.hyp}
|m|\le L^{\chi(z)}.
\end{equation}
Then
\[
z^mP^*(y_1,\ldots,y_k;L^{\chi(y)})\subset P^*(y_1,\ldots,y_k;(O_{k,s}(1)L)^{\chi(y)})z^m.
\]
\end{lemma}
\begin{proof}
Let $p$ be an arbitrary element of $P^*(y_1,\ldots,y_k;L^{\chi(y)})$, viewed in the natural way as a string in the $y_i$ and their inverses.
Repeated application of the identity
$uv=[u^{-1},v^{-1}]vu$
implies that there is some formal string $p'$ obtained by inserting into $p$, at various places, commutators of the form
\begin{equation}\label{eq:extra.comm.1}
[z^{-m},y_j^{-1}]
\end{equation}
or
\begin{equation}\label{eq:extra.comm.2}
[z^{-m},y_j],
\end{equation}
such that
\[
z^mp=p'z^m
\]
as group elements. By \cite[Proposition B.3]{tointon}, for each $j$ there are natural numbers $t_j$ and $n_{j,1},\ldots,n_{j,t_j}$ with
\begin{equation}\label{eq:t.j.bound}
t_j\le O_s(1),
\end{equation}
and generalised commutators $\zeta_{j,1},\ldots,\zeta_{j,t_j}$ and $\hat\zeta_{j,1},\ldots,\hat\zeta_{j,t_j}$ in $z$ and $y_j$, such that
\[
[z^{-m},y_j^{-1}]=\zeta_{j,1}^{n_{j,1}}\cdots\zeta_{j,t_j}^{n_{j,t_j}}
\]
and
\[
[z^{-m},y_j]=\hat\zeta_{j,1}^{n_{j,1}}\cdots\hat\zeta_{j,t_j}^{n_{j,t_j}}.
\]
Moreover, writing $\omega_{j,i}$ for the $z$-weight of the generalised commutator $\zeta_{j,i}$ considered as a commutator in $z$ and $y_j$, the bound (\ref{eq:m.hyp}) implies that the integers $n_{j,i}$ arising from \cite[Proposition B.3]{tointon} satisfy
\begin{equation}\label{eq:m.j.i.bound}
|n_{j,i}|\le L^{\omega_{j,i}\chi(z)}.
\end{equation}
We may therefore replace each instance of (\ref{eq:extra.comm.1}) in $p'$ with a string of the form $\zeta_{j,1}^{n_{j,1}}\cdots\zeta_{j,t_j}^{n_{j,t_j}}$, and each instance of (\ref{eq:extra.comm.2}) with a string of the form $\hat\zeta_{j,1}^{n_{j,1}}\cdots\hat\zeta_{j,t_j}^{n_{j,t_j}}$. Write $p''$ for the string produced by making each of these replacements, so that $p'$ and $p''$ are equal as group elements.

Now, each instance of (\ref{eq:extra.comm.1}) or (\ref{eq:extra.comm.2}) appearing in $p'$ arises from an element of the form $y_j$ or $y_j^{-1}$ in the string $p$. There are at most $L^{\chi(y_j)}$ such elements in the string $p$ by definition, and so for each $j$ there can be at most $L^{\chi(y_j)}$ instances of (\ref{eq:extra.comm.1}) and (\ref{eq:extra.comm.2}) between them. In particular, there are at most $L^{\chi(y_j)}|n_{j,i}|$ instances of the strings $\zeta_{j,i}$ and $\hat\zeta_{j,i}$ in the string $p''$. The inequality (\ref{eq:m.j.i.bound}) therefore implies that the total number of instances of the strings $\zeta_{j,i}$ and $\hat\zeta_{j,i}$ in the string $p''$ is at most $L^{\chi(\zeta_{j,i})}$. Here, $\chi(\zeta_{j,i})$ refers to the weight of $\zeta_{j,i}$ as a generalised commutator in the $x_i$.

Since the list $z,y_1,\ldots,y_k$ is nilcomplete, each $\zeta_{j,i}$ is equal to some $y_{l_{j,i}}$, and each $\hat\zeta_{j,i}$ is equal to some $y_{\hat{l}_{j,i}}$. Let $\hat p$ be the string of $y_l$ formed from $p''$ by replacing each instance of $\zeta_{j,i}$ in $p''$ by $y_{l_{j,i}}$, and each instance of $\hat\zeta_{j,i}$ by $y_{\hat{l}_{j,i}}$. It follows from (\ref{eq:t.j.bound}) that the total number of $\zeta_{j,i}$ and $\hat\zeta_{j,i}$ is at most $O_{k,s}(1)$, and so the preceding paragraph and the definition of $p$ imply that the string $\hat p$ contains at most $(O_{k,s}(1)L)^{\chi(y_l)}$ instances of $y_l$ and its inverse. In particular, $\hat p\in P^*(y_1,\ldots,y_k;(O_{k,s}(M)L)^{\chi(y)})$, and so the lemma follows from the fact that $\hat p$ and $p'$ are equal as group elements.
\end{proof}

We can now finally prove Proposition \ref{prop:translates'}.

\begin{proof}[Proof of Proposition \ref{prop:translates'}]
We prove that there is a set $X$ of size bounded in terms of $s$ and $M$ such that
\[
P_\textup{ord}(y_1,\ldots,y_k;(ML)^{\chi(y)})\subset P_\textup{ord}(y_1,\ldots,y_k;L^{\chi(y_1)},(O_{k,s}(M)L)^{\chi(y_2)},\ldots,(O_{k,s}(M)L)^{\chi(y_k)})X;
\]
by induction on $k$ this suffices to prove the proposition. More precisely, we prove that this holds with $X=P(y_1^{L^{\chi(y_1)}};M^{|\chi(y_1)|})$, which has cardinality at most $2M^s+1$. To prove that it holds with this choice of $X$, note simply that
\[
\begin{split}
P_\textup{ord}(y_1,\ldots,y_k;(ML)^{\chi(y)})\qquad\qquad\qquad\qquad\qquad\qquad
\qquad\qquad\qquad\qquad\qquad\qquad\qquad\qquad\qquad\qquad\\
\qquad\qquad\subset P(y_1,L^{\chi(y_1)})P^*(y_2,\ldots,y_k;(O_{k,s}(M)L)^{\chi(y_2)},\ldots,(O_{k,s}(M)L)^{\chi(y_k)})P(y_1^{L^{\chi(y_1)}};M^{|\chi(y_1)|})
\end{split}
\]
by Lemma \ref{lem:comm.p*}, and that
\[
\begin{split}
P^*(y_2,\ldots,y_k;(O_{k,s}(M)L)^{\chi(y_2)},\ldots,(O_{k,s}(M)L)^{\chi(y_k)})
\qquad\qquad\qquad\qquad\qquad\qquad\qquad\qquad\\
\qquad\qquad\qquad\qquad\qquad\qquad\qquad\qquad
\subset P_\textup{ord}(y_2,\ldots,y_k;(O_{k,s}(M)L)^{\chi(y_2)},\ldots,(O_{k,s}(M)L)^{\chi(y_k)})
\end{split}
\]
by Lemma \ref{lem:split.prog}. This ends the proof of Proposition \ref{prop:translates'}, and hence of Proposition \ref{powers} (\ref{prop:translates}).
\end{proof}

Having proved Proposition \ref{powers}, we can now easily deduce Theorem \ref{prop:scales.nilp}.

\begin{proof}[Proof of Theorem \ref{prop:scales.nilp}.] First we claim that it is enough to find positive integers $N_0$ and $C$ depending only on $K$ and $s$ such that $S^{N_0n}$ is a $C$-approximate subgroup for all $n \geq 1$. Indeed, since $S$ is a $K$-approximate subgroup, there is some $Y \subset G$, the size of which is bounded in terms of $K$ and $s$, such that $S^{2N_0} \subset SY$. Given any integer $m\ge N_0$, there exists $n\ge1$ such that $nN_0 \leq m < (n+1)N_0$, and so $S^{2m} \subset S^{2(n+1)N_0-1} = S^{2nN_0-1}S^{2N_0} \subset S^{2nN_0}Y \subset S^{nN_0}XY$ for some subset $X$ of size at most $C$. For $m\le N_0$, meanwhile, we have $S^{2m}\subset S^{2N_0}\subset S^mY$. This proves the claim.
 
Now, without loss of generality, we may assume that $S$ generates $G$. By Theorem \ref{thm:nilp} there is a finite subgroup $H \leq G$ normalised by $S$, and hence normal in $G$; integers $r ,M\leq K^{O_s(1)}$; and a nilcomplete progression $\overline{P}(L):=\overline{P}(x_1,...,x_r;L)$ for some elements $x_1,...,x_r$ in $G$ such that
$$S \subset H\overline{P}(L) \subset S^{M}.$$
Increasing $M$ if necessary, Proposition \ref{powers} (\ref{prop:dilates.1}) implies that $\overline{P}(nL) \subset \overline{P}(L)^{Mn}$ for all $n \ge 1$. Let $N_0:=M^2$.
Since $H$ is normal in $G$, we have $(H\overline{P}(L))^n=H\overline{P}(L)^n$ for every $n \ge 1$. Applying Proposition \ref{powers} (\ref{prop:dilates.2}) and (\ref{prop:translates}), we therefore see that for each $n \ge 1$ there is a subset $X_n$ of size bounded in terms of $s$ and $K$ such that
\[
\begin{split}
S^{2N_0n} \subset (H\overline{P}(L))^{2N_0n} = H\overline{P}(L)^{2N_0n} \subset H\overline{P}(2N_0nL)
\qquad\qquad\qquad\qquad\qquad\qquad\qquad\\
\subset\overline{P}(nL)X_n\subset H(\overline{P}(L))^{Mn}X_n\subset S^{M^2n}X_n=S^{N_0n}X_n,
\end{split}
\]
and so the theorem is proved.
\end{proof}
This also completes the proof of  Theorem \ref{thm:scales}.

\subsection{Doubling of the various types of progression}
We have made heavy use of the principle that sets of small doubling are controlled by nilprogressions. In this subsection we take the opportunity to explain how Proposition \ref{powers} can be used to shed light on the extent to which the converse holds.

First, we note that nilcomplete progressions always enjoy small doubling. Indeed, an immediate consequence of Proposition \ref{powers} is that for every nilcomplete progression $\overline{P}$ of rank $r$ and step $s$ there is a set $X$ of cardinality $O_{r,s}(1)$ such that $\overline{P}^2\subset\overline{P}X$.

For nilprogressions the situation is not quite so straightforward; one may check, for example, that in the free $s$-step nilpotent group generated by $x_1$ and $x_2$ the ratio $|P^*(x_1,x_2;L,1)^3|/|P^*(x_1,x_2;L,1)|$ tends to infinity as $L\to\infty$. Nonetheless, it is shown in Tao's book \cite[Proposition 2.2.3]{tao.hilb} that, provided the lengths $L_1,\ldots,L_r$ are large enough in terms of the rank $r$ and step $s$ only, a nilprogression is an $O_{r,s}(1)$-approximate group. We recover this fact, and obtain a similar statement for nilpotent progressions, as follows.
\begin{corollary}
Fix $r,s\ge1$. Then there is a constant $\lambda=\lambda_{r,s}$ such that whenever $x_1,\ldots,x_r$ are elements in an $s$-step nilpotent group and $L=(L_1,\ldots,L_r)$ is a vector of integers that are at least $\lambda$, there exists a set $X$ of cardinality at most $O_{r,s}(1)$ such that the nilprogression $P^*=P^*(x_1,\ldots,x_r;L)$ and the nilpotent progression $P=P(x_1,\ldots,x_r;L)$ satisfy
\[
(P^*)^2\subset P^2\subset P^*X\subset PX.
\]
In particular, since it is symmetric the nilprogression $P^*$ is an $O_{r,s}(1)$-approximate group.
\end{corollary}
\begin{proof}
Note that (\ref{eq:PinB}) implies the existence of a constant $C=C_{r,s}$ such that
\begin{equation}\label{eq:bar.in.nilprog.C}
\overline{P}(x_1,\ldots,x_r;L')\subset P^*(x_1,\ldots,x_r;CL')
\end{equation}
for every $L'=(L'_1,\ldots,L'_r)$. We claim that taking $\lambda=2C$ satisfies the requirements of the corollary; thus we may assume that $L_1,\ldots,L_r\ge2C$. This implies in particular that
\begin{equation}\label{eq:L>2C}
\lceil L_i/2C\rceil\le L_i/C
\end{equation}
for each $i$. Denote by $\lceil L/2C\rceil$ the vector $(\lceil L_1/2C\rceil,\ldots,\lceil L_r/2C\rceil)$, and let $X$ be the set obtained by applying Proposition \ref{powers} \eqref{prop:translates} to $\overline{P}(x_1,\ldots,x_r;\lceil L/2C\rceil)$ in the case $M=4C$, so that
\begin{equation}\label{eq:specific.translates}
\overline P(x_1,\ldots,x_r;4C\lceil L/2C\rceil)\subset\overline{P}(x_1,\ldots,x_r;\lceil L/2C\rceil)X.
\end{equation}
Now note that
\begin{align*}
P^*(x_1,\ldots,x_r;L)^2&\subset P(x_1,\ldots,x_r;L)^2                              &&\qquad\text{by (\ref{eq:nested.progs})}\\
                       &\subset \overline{P}(x_1,\ldots,x_r;L)^2                   &&\qquad\text{by (\ref{eq:nested.progs})}\\
                       &\subset \overline{P}(x_1,\ldots,x_r;2L)                    &&\qquad\text{by Proposition \ref{powers}}\\
                       &\subset \overline{P}(x_1,\ldots,x_r;4C\lceil L/2C\rceil)   \\
                       &\subset \overline{P}(x_1,\ldots,x_r;\lceil L/2C\rceil)X    &&\qquad\text{by (\ref{eq:specific.translates})}\\
                       &\subset P^*(x_1,\ldots,x_r;L)X                             &&\qquad\text{by (\ref{eq:bar.in.nilprog.C}) and (\ref{eq:L>2C})}\\
                       &\subset P(x_1,\ldots,x_r;L)X                               &&\qquad\text{by (\ref{eq:nested.progs})}.
\end{align*}
\end{proof}
One may show by a similar argument, making use of Remark \ref{rem:proper}, that a \emph{proper} $s$-step nilpotent progression $P$ of rank $r$ satisfies $|P^k|\le O_{k,r,s}(1)|P|$, irrespective of its side lengths. We leave the details to the interested reader.

\section{A Gromov-type theorem for finite groups}\label{sec:diam}
In this section we prove Theorem \ref{thm:gromovfinite} about the algebraic structure of almost flat Cayley graphs, together with two variants of that statement, which we gather as follows.

\begin{theorem}[Structure of almost flat finite groups]\label{struct} Let $\varepsilon, \delta>0$. Let $G$ be a finite group generated by a symmetric subset $S$ whose Cayley graph has diameter $\gamma:=\diam_S(G)$ and is $\eps$-almost flat, which is to say that
 $$\gamma \geq \left( \frac{|G|}{|S|} \right)^\eps.$$
Then there is a constant $D_{\eps,\delta}$ depending only on $\eps,\delta$ such that if $\gamma \geq D_{\eps,\delta}$ then
 \begin{enumerate}
\item $G$ has a normal subgroup $H$ contained in $S^{[\gamma^\delta]}$ such that $G/H$ has a nilpotent subgroup of index, rank and nilpotency class bounded in terms of $\eps$ and $\delta$ only;
\item $G$ has a normal subgroup $H$ contained in $S^{[\gamma^{1/2 + \delta}]}$ such that $G/H$ has an abelian subgroup of index and rank bounded in terms of $\eps$ and $\delta$ only;
\item $G$ has a normal subgroup $G_0$ of index at most $C_\eps$ that admits a cyclic quotient with diameter at least $c_\eps\gamma$, where $C_\eps,c_\eps$ are positive constants depending on $\eps$ only.
\end{enumerate}
\end{theorem}
In $(3)$, the diameter of the cyclic quotient is understood with respect to a system of generators of $G_0$ derived from $S$ by the usual Reidemeister--Schreier process, which we summarise in the following lemma.
\begin{lemma}\label{lem:rei-schr}
Let $G$ be a group generated by a finite symmetric set $S$, and let $G_0$ be a subgroup of $G$ of index $d$. Then $G_0\cap S^{2d-1}$ contains a symmetric generating set $S_0$ for $G_0$ satisfying $\frac{1}{d}|S|\le|S_0|\le d|S|$ and $\frac{1}{2d}(\diam_S(G)-d)\le\diam_{S_0}(G_0)\le\diam_S(G)$.
\end{lemma}
\begin{proof}
Let $T$ be a set of right-coset representatives for $G_0$ in $G$, and for $x\in G$ write $\tau(x)$ for the unique element of $T$ such that $x\in G_0\tau(x)$. It is shown in \cite[Lemma 7.2.2]{hall} that the set $S_0=\{ts\tau(ts)^{-1}:s\in S,t\in T\}$ is symmetric and generates $G_0$ with diameter at most $\diam_S(G)$. We may assume that $d>1$, in which case $S$ intersects at least two cosets of $G_0$. The argument of Lemma \ref{lem:exhausts.cosets} therefore shows that we may take $T\subset S^{d-1}$, and so $S_0\subset G_0\cap S^{2d-1}$ and $\diam_S(G)\le2d\diam_{S_0}(G_0)+d$, as desired. It is trivial that $|S_0|\le d|S|$, and the fact that $|S_0|\ge \frac{1}{d}|S|$ follows because there must be a coset of $G_0$ containing at least $\frac{1}{d}|S|$ elements of $S$.
\end{proof}
\begin{remark}\label{rem:rei-schr}
Lemma \ref{lem:rei-schr} implies that if the index of $G_0$ is bounded and $\diam_S(G)$ is sufficiently large then $(G,S)$ is almost flat if and only if $(G_0,S_0)$ is almost flat.
\end{remark}

\begin{remark}Theorem \ref{struct} $(3)$ can be seen as a finite analogue of the key step in Gromov's proof of his polynomial growth theorem \cite{gromov} in which he establishes that every finitely generated group with polynomial growth has a finite index subgroup, which maps onto $\Z$.
\end{remark}

\begin{remark} Lee and Makarychev \cite{lee-makarychev} show that if the Cayley graph of a finite group is uniformly doubling, then a similar conclusion to that of Theorem \ref{struct} (3) holds. Their proof is very different and closer to Kleiner's approach to Gromov's theorem via harmonic functions, and requires the doubling assumption to hold at all scales. By contrast, our argument requires doubling only at one large scale (see Prop. \ref{doubling-structure}).
\end{remark}

\begin{remark} The proof of Theorem \ref{struct} $(2)$ shows, more generally, that for any integer $s \geq 1$ the group $G$ has a normal subgroup $H$ contained in $S^{[\gamma^{\frac{1}{s+1}+\delta}]}$ such that $G/H$ has an $s$-step nilpotent subgroup of index and rank bounded in terms of $\eps$ and $\delta$ only.
\end{remark}

We also prove the following result, which may be thought of as a converse to Theorem \ref{struct}.

\begin{prop}[Nilpotent groups are almost flat]\label{prop:ab.converse}
Let $G$ be a finite $s$-step nilpotent group generated by a symmetric subset $S$. There is a positive constant $\eps=\eps(s,|S|)$ depending only on $s$ and $|S|$ such that $\diam_S(G)\ge|G|^\eps$.
\end{prop}
\begin{remark}
The example of $\F_2^n$ with the standard generating set shows that the $\eps$ in Proposition \ref{prop:ab.converse} has to depend on $|S|$.
\end{remark}

\bigskip

We start the proof of Theorem \ref{struct} by translating the almost flat assumption into a doubling property.

\begin{lemma}[Almost flat groups are doubling at some scale]\label{lem:5adic.ph}
Let $\eps,\delta>0$ and set $K=5^{\frac{2}{\eps \delta}}$. Consider an $\eps$-almost flat Cayley graph of a finite group  $G$ with symmetric generating set $S$. Let $\gamma:=\diam_S(G)$ be its diameter. Then there is an integer $n$ satisfying $\gamma^{\delta/2} \leq n \leq \gamma^\delta$ such that
 $$|S^{5n}|\leq K|S^n|.$$
\end{lemma}
\begin{proof}
Let $k_0$ be the largest $k$ such that $5^k\gamma^{\delta/2} \leq 5\gamma^\delta$, and note therefore that $5^{k_0} \geq \gamma^{\delta/2}$.
The almost flat assumption therefore implies that
\[
|S^{5^{k_0}\gamma^{\delta/2}}|\le|G|\le|G||S^{\gamma^{\delta/2}}|/|S|\le\gamma^{1/\eps}|S^{\gamma^{\delta/2}}|=\left(\gamma^{\delta/2}\right)^{2/\delta\eps}|S^{\gamma^{\delta/2}}|\le5^{2k_0/\delta\eps}|S^{\gamma^{\delta/2}}|=K^{k_0}|S^{\gamma^{\delta/2}}|,
\]
from which the conclusion follows easily.
\end{proof}


\begin{proof}[Proof of Theorem \ref{struct} (1)]  Let $K=5^{\frac{4}{\eps \delta}}$ and let $n_0=n_0(K)$ be the integer given by Proposition \ref{prop:v.nilp}, and take $D$ large enough to force $\gamma^{\delta/4}>n_0$. Lemma \ref{lem:5adic.ph} then implies that there exists $n$ with $n_0 \leq n \leq \gamma^{\delta/2}$ such that $|S^{5n}|\leq K|S^n|.$ The conclusion then follows from Proposition  \ref{prop:v.nilp}.
\end{proof}

We now pass to the proof of part $(2)$ of Theorem \ref{struct}. It relies on some structural results about finite nilpotent groups, and in particular about commutators and nilprogressions; we refer the reader to Section \ref{sec:coms.nilprogs} for background on these and definitions. We will also use the following standard fact, which follows from \cite[Proposition B.2]{tointon}, for example.
\begin{lemma}\label{multilin}
Let $\alpha$ be a \emph{commutator form of weight $k$} in the sense of \cite[Definition 3.2]{tointon}. Then the value of $\alpha(y_1,\ldots,y_k)$ mod $G_{k+1}$ is multilinear in the variables $y_1,\ldots,y_k$.
\end{lemma}
This means in particular that the value of $\alpha(y_1,\ldots,y_k)$ mod $G_{k+1}$ depends only on the values of $y_i$ mod $[G,G]$.

\begin{lemma}[Commutators in finite nilpotent groups are close to the identity] \label{small-commutator} Let $G$ be a finite $s$-step nilpotent group generated by the nilprogression $P=P^*(x_1,\ldots,x_r;L)$. Then the commutator subgroup $[G,G]$ is contained in $P^{O_{r,s}(\gamma^{1/2})}$, where $\gamma:=\diam_P(G)$.
\end{lemma}

\begin{proof}By induction on $s$, it is enough to prove that $G_s$, the last non-trivial term in the lower central series of $G$, is contained in $P^{O_{r,s}(\gamma^{1/2})}$. In fact, since $G_s$ is an abelian group generated by the basic commutators of weight $s$ in the $x_i$, of which there are at most $O_{r,s}(1)$, it suffices to prove that every power of each such commutator is contained in $P^{O_{r,s}(\gamma^{1/2})}$. However, Lemma \ref{multilin} means that if $\alpha(x_{j_1},...,x_{j_s})$ is a commutator of weight $s$ in the $x_i$ then $\alpha(x_{j_1},...,x_{j_s})^n=\alpha(x_{j_1}^n,x_{j_2},...,x_{j_s})$. We have $x_{j_1}^n\in P^\gamma$ by definition, and so there are $\ell_1,\ldots,\ell_r$ such that $|\ell_i|\le L_i\gamma$ and such that $x_i^n=x_1^{\ell_1}\cdots x_r^{\ell_r}$ (mod $[G,G]$), and hence that $\alpha(x_{j_1}^n,x_{j_2},...,x_{j_s})=\prod_{i=1}^r\alpha(x_i,x_{j_2},...,x_{j_s})^{\ell_i}$, again by Lemma \ref{multilin}. It will therefore suffice to show that each $\alpha(x_i,x_{j_2},...,x_{j_s})^{\ell_i}$ with $|\ell_i|\le L_i\gamma$ belongs to $P^{O_{r,s}(\gamma^{1/2})}$. However, $\alpha(x_i,x_{j_2},...,x_{j_s})^{\ell_i}\in\overline P(x_1,\ldots,x_r;\lceil\gamma^{1/s}\rceil L)$ by definition, and this last set is contained in $P^{O_{r,s}(\gamma^{1/s})}$ by the proof of Proposition \ref{powers} \eqref{prop:dilates.1}, since in the notation of that proof we have $B(L)\subset P^*(L)$. The lemma is therefore proved.
\end{proof}
\begin{remark}
Essentially the same proof shows that the same result holds if $P$ is a nilpotent progression or nilcomplete progression.
\end{remark}

\begin{proof}[Proof of Theorem \ref{struct} (2)]  By part (1) applied to $\delta=\frac{1}{3}$, we may pass to a quotient and assume that $G$ has a normal nilpotent subgroup $\Gamma$ whose index, class and rank are bounded in terms of $\eps$ only. Indeed, the ratio $|G|/|S|$ does not increase when passing to a group quotient, and the diameter of the quotient will be at least $\gamma-\gamma^{1/3}$. Lemma \ref{lem:rei-schr} then gives a generating set $S_0\subset S^{O_\eps(1)}$ for $\Gamma$ such that $\diam_S(\Gamma)$ is comparable to $\gamma$ and such that $(\Gamma,S_0)$ is $\eps'$-almost flat for some positive $\eps'$ depending on $\eps$ only. Arguing as in the proof of Theorem \ref{struct} (1), and using Lemma \ref{lem:sm.doub.app.gp}, we see that there exists $n\le O_{\eps,\delta}(\gamma^{\delta})$ such that $S_0^n$ is an $O_{\eps,\delta}(1)$-approximate subgroup. By Theorem \ref{thm:nilp}, therefore, and possibly after passing to a further quotient, there exists a nilprogression $P$ of bounded rank such that $S_0\subset P\subset S_0^{O_{\eps,\delta}(\gamma^{\delta})}$, and hence that $\diam_P(\Gamma)\ge\Omega_{\eps,\delta}(\gamma^{1-\delta})$. Lemma \ref{small-commutator} therefore implies and we conclude that $[\Gamma,\Gamma] \subset P^{O_{\eps}(\gamma^{1/2})}$, and hence $[\Gamma,\Gamma] \subset S^{O_{\eps,\delta}(\gamma^{1/2+\delta})}$. The subgroup $H:=[\Gamma,\Gamma]$ is characteristic in $\Gamma$, and hence normal in $G$, and the quotient $G/H$ has the desired properties.
\end{proof}

\begin{lemma}\label{lem:ab.quotient}
Let $G$ be a finite nilpotent group of rank $r$ and step $s$. Then
\[
|G/[G,G]|\ge|G|^{\Omega_{r,s}(1)}.
\]
\end{lemma}
\begin{proof}
We denote by $G=G_0=G_1\ge G_2\ge\ldots\ge G_s\ge G_{s+1}=\{1\}$ the lower central series of $G$. Observe that
\begin{equation}\label{eq:lcs}
|G_k|=\prod_{i=k}^s|G_i/G_{i+1}|,
\end{equation}
that $G_k/G_{k+1}$ is abelian and generated by the basic commutators in the $x_i$ of total weight $k$ (see Section \ref{sec:coms.nilprogs} or \cite{hall}). Let $x_1,\ldots,x_r$ be a generating set for $G$. Note that the group $G/[G,G]$ is generated by the elements $x_i[G,G]$. Write $d$ for the order of the element $x_i[G,G]$ of maximum order in $G/[G,G]$, and note that
$|G/[G,G]|\ge d$. It follows from Lemma \ref{multilin} that the order in the group $G_k/G_{k+1}$ of a commutator of total weight $k$ featuring generators $x_{i_1},\ldots,x_{i_k}$ is at most $d$. Therefore, writing $b(r,k)$ for the number of basic commutators of total weight $k$ in $r$ generators we have
\[
|G_{k-1}/G_k|\le d^{b(r,k)}.
\]
In particular, writing $c(r,s)$ for the number of basic commutators of total weight at most $s$ in $r$ generators, (\ref{eq:lcs}) implies that
$|[G,G]|\le d^{c(r,s)}$. The result follows.
\end{proof}

\begin{proof}[Proof of Proposition \ref{prop:ab.converse}] Let $\gamma:=\diam_S(G)$. Suppose first that $G$ is abelian. Then $|S^n| \leq n^{|S|}$, and hence $\gamma \geq |G|^{1/|S|}$. If $G$ is nilpotent of class $s$, then Lemma \ref{lem:ab.quotient} shows that $|G| \leq |G/[G,G]|^{C_{|S|,s}} \leq \gamma^{|S|C_{|S|,s}}$ as desired.
\end{proof}

We now pass to the proof of the part $(3)$ of Theorem \ref{struct}. For this we require the following fact about flat tori.

\begin{lemma}[Euclidean tori have large one-dimensional quotients]\label{torus-quotient} Let $\Delta$ be a discrete subgroup of rank $n$ in $\R^n$ and consider the torus $T:=\R^n/\Delta$, endowed with the standard Euclidean metric. Then there is a one-dimensional quotient torus with diameter at least $\frac{\pi}{2nn!}$ times the diameter of $T$.
\end{lemma}

\begin{proof} Let $v_1,...,v_n$ be a basis of successive minima of $\Delta$, with $\|v_1\| \leq ... \leq \|v_n\|$. Let $H$ be the $\R$-span of $v_1,...,v_{n-1}$. Let $T_0$ be the quotient torus $(\R^n/H)/(\Delta \cap H)$. The diameter of $T_0$ is the supremum over all $x \in \R^n$ of the distance from $x$ to $H+\Delta$. This is half the distance from $v_n$ to $H$. On the other hand the diameter $\diam(T)$ of $T$ is the supremum  $\sup_{x  \in \R^n}\inf_{\delta \in \Delta} \|x-\delta\|$, which is at most $\frac{1}{2}\sum_1^n\|v_i\| \leq \frac{n}{2}\|v_n\|$.
On the other hand, it follows from Minkowski's second theorem that if $V_n$ denotes the volume of the Euclidean $n$-ball, $$2^n/n!\leq \frac{V_n \|v_1\|...\|v_n\|}{\|v_1 \wedge ... \wedge v_n \|} \leq 2^n$$ and similarly for $v_1 \wedge ... \wedge v_{n-1}$. Since $\|v_1 \wedge ... \wedge v_{n-1}\|d(v_n,H)=\|v_1 \wedge ... \wedge v_{n}\|$, we conclude that $\|v_n\|\leq 2(n-1)! d(v_n,H) V_{n-1}/V_n$. On the other hand, it is well known that $V_{n-1}/V_n \leq n/\pi$, and so the result follows.
\end{proof}

\begin{proof}[Proof of Theorem \ref{struct} (3)] Pick $\delta=1/6$. By part $(2)$ of the theorem there is a normal subgroup $H$ contained in $S^{O_{\eps}(\gamma^{2/3})}$ such that $G/H$ has an abelian subgroup of rank and index bounded in terms of $\eps$ only. Applying Lemma \ref{lem:rei-schr}, we may pass to a subgroup of bounded index and assume that $G/H$ is abelian of bounded rank. Lemma \ref{lem:5adic.ph} implies that, provided $\gamma$ is larger than some constant depending on $\eps$ only, there is some $n_1 \leq \gamma^{1/3}$ such that $S^{2n_1}$ is an $O_{\eps}(1)$-approximate subgroup. Its image $\overline{S}^{2n_1}$ in $G/H$ is therefore also an $O_\eps(1)$-approximate group, and so there is a subgroup $H_0\in\overline S^{8n_1}$ and a generalised arithmetic progression $P=P(v_1,...,v_r;L)$ of bounded rank such that $S^{2n_1} \subset H_0P \subset S^{O_\eps(n_1)}$ \cite{green-ruzsa}. On replacing $H$ by the pullback of $H_0$ to $G$, we may in fact assume that $H_0$ is trivial. Note that if $\gamma=\diam_S(G)$ is larger than a constant depending on $\eps$ only, $$2n_1\diam_{P}(G/H) \leq\diam_{\overline{S}}(G/H)\leq O_\eps(n_1\diam_{P}(G/H)).$$
The same comparison holds for the diameter of every group quotient. The generators $v_1,...,v_r$ allow us to write $G/H$ as $\Z^r/\Delta$ for some lattice $\Delta$; thus, we write $x\in G/H$ in coordinate form as $x=(x_1,\ldots,x_r)$. Since $G/H$ is abelian, $P^n=P(v_1,...,v_r;nL)$ for every $n$. Let $\|\,\cdot\,\|$ be the norm on $\R^r$ defined by $\|x\|:=\max_1^r |x_i/L_i|$, and define $\ell(x):=\max_1^r \lceil |x_i/L_i| \rceil$. Note that $\|x\| \leq \ell(x) \leq \|x\| +1$ and that
$$\diam_{P}(G/H) = \sup_{x \in \Z^r} \inf_{\delta \in \Delta} \ell(x-\delta).$$
We may compare this diameter with the diameter of the Euclidean torus $T:=\R^r/\Delta_L$, where $\Delta_L$ is the lattice of all vectors of the form $(x_1/L_1,...,x_r/L_r)$ with $(x_1,...,x_r)\in\Delta$. Indeed,
$$\diam_{P}(G/H) \leq 1+ \sup_{x \in \Z^r} \inf_{\delta \in \Delta} \|x-\delta\| \leq 2+ \sup_{x \in \R^r} \inf_{\delta \in \Delta} \|x-\delta\|,$$
where we have used that $\inf_{y \in \Z^r} \|x-y\|\leq 1$ for all $x \in \R^r$, because $L_i \geq 1$. Hence
$$\diam_{P}(G/H) \leq  2+ \sup_{y \in \R^r} \inf_{\delta \in \Delta_L} \|y-\delta\|_\infty \leq 2+ \sup_{y \in \R^r} \inf_{\delta \in \Delta_L} \|y-\delta\|_2 = 2+ \diam(\R^r/\Delta_L),$$
where $\|\cdot\|_2$ is the standard Euclidean norm on $\R^r$. Lemma \ref{torus-quotient} implies that there is a hyperplane $\mathcal{H}_L$ in $\R^r$ intersecting $\Delta_L$ in a lattice such that the diameter of the one-dimensional torus quotient $\R^r/(\mathcal{H}_L+\Delta_L)$ satisfies $$\diam(\R^r/\Delta_L) \leq \frac{2rr!}{\pi}\diam(\R^r/(\mathcal{H}_L+\Delta_L)).$$
Moreover,
\[
\diam(\R^r/(\mathcal{H}_L+\Delta_L))= \sup_{y \in \R^r}\inf_{\delta\in\mathcal{H}_L+\Delta_L} \|y-\delta\|_2\leq  \sqrt{r} \sup_{y \in \R^r} \inf_{\delta \in \mathcal{H}_L+\Delta_L} \|y-\delta\|_\infty  = \sqrt{r} \sup_{x \in \R^r} \inf_{\delta \in \mathcal{H}+\Delta} \|x-\delta\|,
\]
where $\mathcal{H}=\{(L_1x_1,...,L_rx_r); (x_1,...,x_r) \in \mathcal{H}_L\}$. Writing $\mathcal{H}_\Z:=\mathcal{H}\cap \Z^r$, we then have
$$\sup_{x \in \R^r} \inf_{\delta \in \mathcal{H}+\Delta} \|x-\delta\| \leq 1+ \sup_{x \in \Z^r} \inf_{\delta \in \mathcal{H}_\Z+\Delta} \|x-\delta\| \leq 1+ \sup_{x \in \Z^r} \inf_{\delta \in \mathcal{H}_\Z+\Delta} \ell(x-\delta).$$
However, $\sup_{x \in \Z^r} \inf_{\delta \in \mathcal{H}_\Z+\Delta} \ell(x-\delta)$ is precisely the diameter of the cyclic quotient $Q:=\Z^r/(\mathcal{H}_\Z+\Delta)$ with respect to the image of $P$, and so, combining all of the above, we obtain
$$\diam_{\overline S}(G/H) \leq O_\eps(n_1 \diam_P(G/H)) \leq O_\eps( n_1 \diam_{\overline{P}}(Q)) \leq O_\eps( \diam_{\overline{S}}(Q)).$$
\end{proof}

This ends the proof of Theorem \ref{struct}.

\bigskip

We now record the following simple consequence of Part $(1)$ of Theorem \ref{struct}.

\begin{corollary}[Almost flatness is a group property]\label{cor:intrinsic}
Let $\eps>0$ and suppose that $G$ is a finite group with a symmetric generating set $S$ with respect to which $G$ is $\eps$-almost flat. Suppose that $E$ is some other symmetric generating set for $G$. Then  there is $\eps'=\eps'(\eps,|E|)>0$ and $c_\eps>0$ depending on $\eps$ only such that $\diam_E(G)\geq  c_\eps |G|^{\eps'}$.
\end{corollary}

\begin{proof} By Theorem \ref{struct} $G$ has a quotient $Q$ of comparable diameter containing a nilpotent subgroup of bounded rank, index and nilpotency class, and so the conclusion follows from Proposition \ref{prop:ab.converse} and Lemma \ref{lem:rei-schr}.
\end{proof}
\begin{remark} We caution that, in spite of Corollary \ref{cor:intrinsic}, for every $\eps>0$ there exists $C=C_\eps$ such that every finite group $G$ satisfying $|G|\ge C$ has at least one Cayley graph that is $\eps$-almost flat and at least one Cayley graph that is not $\eps$-almost flat. Indeed, taking $S=G\backslash\{x,x^{-1}\}$ for some $x\in G$ gives an almost flat Cayley graph. On the other hand, as is well known and easy to prove, every finite group $G$ of size at most $2^n$ has a generating set of size at most $n$ and diameter at most $n$.
\end{remark}

\subsection{Explicit bounds, the polynomial Freiman-Ruzsa conjecture and an example}\label{sharpness}
The only source of ineffectiveness in our arguments comes from Theorem \ref{thm:bgt}, the structural result for approximate subgroups of arbitrary groups proved in \cite{bgt}. Any explicit bound in terms of $K$ on the number of translates of $\Gamma$ needed to cover the $K$-approximate subgroup $A$ would yield to an explicit bound on $n_0(K)$ in Theorem \ref{thm:scales} and on the index of the abelian subgroup in Theorem \ref{thm:gromovfinite}. In this subsection we present an example to show that
\begin{enumerate}
\item Theorem \ref{thm:gromovfinite} is sharp in the sense that the power $\eps$ cannot be allowed to tend to zero without losing the conclusion that the index of the abelian subgroup remains bounded; and
\item the bound on the number of translates required in Theorem \ref{thm:bgt}, and hence the quantity $n_0(K)$ in Theorem \ref{thm:scales}, is worse than polynomial in $K$.
\end{enumerate}

We give more precise versions of these statements below as Facts \ref{fact1} and \ref{fact2}, as well as giving more detail on the context in which they should be viewed. However, first we establish some of the notation of our example.

Let $n$ be an even integer $\geq 8$, and let $p=p(n)$ be a prime such that $p>n$. We define groups $L_n\ge G'_n\ge G_n$ as follows. First, take
\[
L_n=Sym(n) \ltimes \F_p^n,
\]
with the semidirect product defined by the natural action of $Sym(n)$ on the basis vectors, and write $\pi:L_n\to Sym(n)$ for the natural projection. Writing $V$ for the subspace of $\F_p^n$ consisting of those vectors whose coordinates sum to zero, define the subgroup $G_n'$ of $L_n$ by
\[
G_n'=Sym(n)\ltimes V.
\]
Finally, define the subgroup $G_n$ of index $2$ in $G_n'$ by
\[
G_n=Alt(n)\ltimes V.
\]

Let $\tilde{S}$ be the generating set of $L_n$ consisting of the identity and the three elements $(c;0)$, $(\tau;0)$, $(1;1,0,\ldots,0)$ and their inverses. Here $c$ is the long cycle $(12\ldots n)$ and $\tau$ is the transposition $(12)$. It is clear that $L_n$ is the direct product of $G'_n$ with the central subgroup $Z \cong \Z/p\Z$ made of elements of the form $(1,v)$ in which all coordinates of $v$ are equal, and so we may set $S'$ to be the image of $\tilde S$ under projection to $G_n'$. Finally, let $S$ be the generating set for $G_n$ obtained from $S'$ using Lemma \ref{lem:rei-schr}, and note that $|S|\le8$.

The reason for making these definitions is that the groups they describe satisfy the following facts.

\begin{fact}\label{fact1}For every function $f:\N \to (0,+\infty)$ going to zero at infinity there exist choices of $p(n)$ such that $\diam_S(G_n)\ge |G_n|^{f(|G_n|)}$ but such that no subgroup of $G_n$ of index $<n/2$ admits a non-trivial abelian quotient.
\end{fact}

\begin{fact}\label{fact2} For every function $f:\N \to \N$, there are a choice of $p(n)$, an increasing sequence of integers $K_n$, and $K_n$-approximate subgroups $A_n$ in $L_n$ such that, for large enough $n$, the approximate subgroup $A_n$ is not contained in fewer than $K_n^{\frac{1}{200}\log^{(4)} K_n}$ cosets of any subgroup $\Gamma \leq L_n$ such that $\Gamma$ has a normal subgroup $H$ contained in $A_n^{f(K_n)}$ with $\Gamma/H$  solvable.
\end{fact}
Here we have denoted by $\log^{(4)}$ the $4$th iteration of $\log$ in base $e$.

Fact \ref{fact1} establishes statement $(1)$ above, while Fact \ref{fact2} corresponds to statement $(2)$. Fact \ref{fact2} should be compared with the classical case of approximate subgroups of abelian groups, where it is expected that only $O(K^{O(1)})$ translates of some coset-progression are needed to cover the approximate subgroup (this is the so-called \emph{polynomial Freiman-Ruzsa type conjecture}, see \cite{green-open-pb} and references therein). In the abelian case bounds of the form $e^{(\log K)^{O(1)}}$ have been established by Sanders \cite{sand.survey}.
\begin{remark*}
After a first draft of this paper circulated, Eberhard \cite{eberhard} adapted the construction of Fact \ref{fact2} to show that for arbitrarily large $K$ there exists a group $G$ with a finite $K$-approximate subgroup that is not covered by $K^{\log\log K}$ cosets of any subgroup $\Gamma$ having a finite normal subgroup $H$ with $\Gamma/H$ solvable. This nicely improves the bound in our Fact \ref{fact2}. Note that the example in \cite{eberhard} generates an infinite group, while we work here in the setting of finite groups.
\end{remark*}

\bigskip

Before we prove Facts \ref{fact1} and \ref{fact2}, let us isolate a few helpful statements about the groups $L_n$, $G_n'$ and $G_n$. We start by bounding their diameters, which we denote from now on by
\[
\gamma=\diam_{\tilde S}(L_n),\qquad\qquad\gamma'=\diam_{S'}(G_n'),\qquad\qquad\gamma_0=\diam_S(G_n).
\]
\begin{lemma}\label{lem:diam.LGG}
We have
\begin{enumerate}
\item $\frac{1}{2}(p-1)\le\gamma\le np+O(n^2)$;
\item $\frac{1}{2}(p^{1-1/n}-1)\le\gamma'\le np+O(n^2)$;
\item $\frac{1}{10}p^{1-1/n}\le\gamma_0\le np+O(n^2)$.
\end{enumerate}
\end{lemma}
\begin{proof}
Item $(3)$ follows from item $(2)$ by Lemma \ref{lem:rei-schr}. The upper bound of $(2)$ trivially follows from the upper bound of $(1)$. To prove the lower bound of $(1)$ note that the elements of $\tilde{S}^R$ are all in $Sym(n) \times [-R,R]^n$, and hence that $(2\gamma+1)^n\ge p^n$. To prove the lower bound of $(2)$, note that the projections of these elements to $G_n'$ lie in $Sym(n) \times B_n$ for some subset $B_n$ of $V$ with size at most $(2R+1)^n$, and hence that $(2\gamma+1)^n \geq p^{n-1}$. For the upper bound of $(1)$, note that $O(n^2)$ steps are enough to cover $Sym(n)$ (see \cite{lubotzky-book}, for example), while $p$ steps are enough to obtain the line $(1;\F_p,0,\ldots,0)$ and thus (conjugating by powers of $c$) we obtain each coordinate line in $p+n$ steps and thus all of $\F_p^n$ in $n(p+n)$ steps. Hence the diameter of $L_n$ with respect to $\tilde{S}$ is at most $np+O(n^2)$.
\end{proof}
\begin{lemma}\label{lem:sm.index}
The only subgroups of $G_n'$ of index $<n$ are $G_n'$ itself and $G_n$.
\end{lemma}
\begin{proof}
It is well known that the only proper subgroup of $Sym(n)$ with index $<n$ is $Alt(n)$, so if $\Gamma$ is a subgroup of $G'_n$ with index $<n$ then $\pi(\Gamma)$ contains $Alt(n)$. Since $p>n$, the projection $\pi$ is not injective in restriction to $\Gamma$. However, $\ker \pi \cap \Gamma$ is invariant under the action of $Alt(n)$. It is well known that $Sym(n)$ and $Alt(n)$ act irreducibly on $V$, and so this implies that $\ker \pi \cap \Gamma=V$, and hence that $Alt(n)\ltimes V\subset\Gamma$.
\end{proof}

\begin{proof}[Proof of Fact \ref{fact1}]
Since $|G_n| \leq  n! p^{n-1} \leq (np)^{n-1} \leq p^{2(n-1)}$, Lemma \ref{lem:diam.LGG} implies that $\gamma_0 \gg |G_n|^{\frac{1}{2n}}$. In light of Lemma \ref{lem:sm.index}, therefore, Fact \ref{fact1} is satisfied upon setting $p=p(n)$ sufficiently large.
\end{proof}

\bigskip
%
%
%
%

We now pass to the proof of Fact \ref{fact2}. We use the following lemma.

\begin{lemma}\label{bochert} Let $\beta\in(0,1)$ and suppose that $(\log \log n)^{-\beta}\le d<1$. Write $t=d^{3/d}$ and $b=9/t^2$. Then, provided $n \geq n(\beta)$, every subgroup $H \leq Sym(n)$ of order at least $d^n n!$ has a subgroup $H_0$ with $[H:H_0] \leq b$ that stabilises a subset $\Omega_0 \subset \{1,\ldots,n\}$ of size at least $\frac{d}{3b}n$ on which it acts as $Sym(\Omega_0)$ or $Alt(\Omega_0)$.
\end{lemma}
We prove Lemma \ref{bochert} shortly. First, let us use it to prove Fact \ref{fact2}.

\begin{proof}[Proof of  Fact \ref{fact2}]
%
By Lemma \ref{lem:diam.LGG} we have $p^{1-\frac{1}{n}} \leq 2\gamma+1 \leq O(np)$. Setting $\eps=\frac{1}{4n}$ and $\delta = \frac{1}{2}$, and taking $K_n=5^{\frac{2}{\eps \delta}} = 5^{16n}$, Lemmas \ref{lem:5adic.ph} and \ref{lem:sm.doub.app.gp} therefore give a scale $k \in [\gamma^{1/4}, \gamma^{1/2}]$ such that $A_n:=\tilde{S}^{2k}$ is a $K_n$-approximate subgroup.

Assume that there are subgroups $H \leq \Gamma$ in $L_n$, such that $H$ is normal in $\Gamma$, contained in $\tilde{S}^{2kf(K_n)}$, $\Gamma/H$ is solvable, and $\tilde{S}^{2k}$ is contained in fewer than $K_n^{ \alpha \log \log \log \log K_n}$ cosets of $\Gamma$ for some fixed $\alpha>0$ to be determined shortly. We may choose $p$ large enough that  $K_n^{ \alpha \log \log \log \log K_n}<\gamma^{1/4} < 2k$, which by Lemma \ref{lem:exhausts.cosets} implies that $\Gamma$ has index at most $K_n^{ \alpha \log \log \log \log K_n}$ in $L_n$. Now set $B=5^{16}$, so that $K_n=B^n$, and set $\alpha:=\frac{1}{3\log B}$.  Note that $\alpha \geq \frac{1}{200}$.

If $n$ is large enough, we then see that the index $K_n^{ \alpha \log \log \log \log K_n}$ is bounded above by $d^{-n}$, where $d=(\log \log n)^{-\frac{1}{2}}$. We may thus apply Lemma \ref{bochert} to the subgroup $\pi(\Gamma) \leq Sym(n)$, because its index is at most $d^{-n}$ with $d=(\log \log n)^{-\frac{1}{2}}$. We get a subgroup $\Gamma_0 \leq \Gamma$ of index at most $b$ such that $\Gamma_0$ fixes a subset $\Omega_0$ of $\{1,\ldots,n\}$ of size at least $\frac{d}{3b}n$, where it acts as $Sym(\Omega_0)$ or $Alt(\Omega_0)$. Note that $b\leq \log n$, and that $\frac{d}{3b}n$ is at least $\frac{n}{(\log n)^2}$, when $n$ is large. Let $H_0=H \cap \Gamma_0$, and $\pi_{\Omega_0}$ the natural projection from $Sym(\Omega_0) \ltimes \F_p^{\Omega_0}$ to $Sym(\Omega_0)$. There is a natural restriction homomorphism from $\Gamma_0$ to $Sym(\Omega_0) \ltimes \F_p^{\Omega_0}$, which consists in forgetting the entries outside of $\Omega_0$. Let $\overline{\Gamma_0}, \overline{H_0}$ be the images of $\Gamma_0$, $H_0$ under this map.

Now observe that if $n$ is large enough $\overline{\Gamma_0}$ must contain $Alt(\Omega_0) \ltimes V_{\Omega_0}$, where  $V_{\Omega_0}$ is the subspace of $\F_p^{\Omega_0}$ of vectors whose coordinates sum to $0$. Indeed, to see this it is enough to show that $|\overline{\Gamma_0}|>p|{\Omega_0}|!$, because then $\pi_{\Omega_0}$ will not be injective on $\Gamma_0$ and there will be an element in the kernel whose coordinates are not all equal, so that the span of its orbit under $Alt({\Omega_0})$ will contain  $V_{\Omega_0}$. On the other hand if $|\overline{\Gamma_0}|\leq p|{\Omega_0}|!$, then $|\Gamma_0| \leq p|{\Omega_0}|! p^{n-|{\Omega_0}|} (n-|{\Omega_0}|)! \leq n! p^n p^{1-|\Omega_0|}$, so that  $p^n n!=|L_n| \leq b |\Gamma_0|[L_n:\Gamma]\leq b n! p^n p^{1-|\Omega_0|} K_n^{\alpha \log^{(4)} n}$, which implies that $p^{|\Omega_0| -1} \leq e^{n \log n}$ if $n$ is large enough. But $|\Omega_0|\geq n/(\log n)^2$ and $p\geq e^{n^2}$, so this is impossible unless $n$ is bounded.

The group $\pi_{\Omega_0}(\overline{H_0})$ is a normal subgroup of $Sym({\Omega_0})$. Hence it is either trivial, or must contain $Alt({\Omega_0})$. If it contains $Alt({\Omega_0})$, then, since $\overline{H_0}$ is normal in $\overline{\Gamma_0}$, and $\overline{\Gamma_0}$ contains $V_{\Omega_0}$, $\overline{H_0}$ must contain $[V_{\Omega_0},Alt({\Omega_0})]=V_{\Omega_0}$, hence $\overline{H_0}$ contains $Alt({\Omega_0}) \ltimes V_{\Omega_0}$. However we observed in the proof of Lemma \ref{lem:diam.LGG}, the vector coordinates in $\tilde{S}^R$ are contained in the interval $[-R,R]$. Since $H_0$ lies in $\tilde{S}^{2kf(K_n)}$ we must have $p \leq 2kf(K_n) = O(f(K_n)\sqrt{np})$, which we may rule out by choosing $p$ large enough in terms of $n$. It follows that $\pi_{\Omega_0}(\overline{H_0})$ is trivial. But this contradicts the fact $\pi_{{\Omega_0}}(\overline{\Gamma_0})/\pi_{\Omega_0}(\overline{H_0})$ is solvable (being a homomorphic image of $\Gamma_0/H_0$), because $Alt({\Omega_0})$ is not solvable if $n$ is large enough. This contradiction ends the proof of Fact \ref{fact2}.
\end{proof}

The proof of Lemma \ref{bochert} requires the following two preliminary steps.

\begin{lemma}\label{largeorb} There is $n_0 \in \N$ such that if $d \in [0,1]$, $n\geq n_0$, and $H \leq Sym(n)$ is a subgroup of order at least $d^n n!$, then $H$ has an orbit of size at least $\frac{d}{3}n$.
\end{lemma}

\begin{proof} Let $n_1,\ldots,n_s$ be the size of the orbits of $H$. We have $H \leq Sym(n_1) \times \ldots \times Sym(n_s)$, and hence $|H| \leq n_1! \ldots n_s! \leq n_1^{n_1} \ldots n_s^{n_s}$. If all the $n_i$ are at most $\frac{d}{3} n$, then we get $|H| \leq d^n (\frac{n}{3})^n$, and hence $n! \leq (\frac{n}{3})^n$, which implies that $n$ is bounded.
\end{proof}

\begin{lemma}\label{largeblock} There is $m_0 \in \N$ such that if $t \in [0,1]$, $m\geq m_0$, and $H \leq Sym(m)$ is a transitive subgroup of order at least $t^m m!$, then the maximal number of imprimitivity blocks  of $H$ is at most $(3/t)^2$.
\end{lemma}

\begin{proof}Let $a$ be the maximal possible number of blocks of imprimitivity. If $H$ is not primitive, then $a$ is a proper divisor of $m$ and thus $a \leq \frac{m}{2}$. Moreover $H$ is contained in the wreath product of $Sym(a)$ with $Sym(\frac{m}{a})$, so that $|H| \leq a! ((\frac{m}{a})!)^a$. Hence
\[
t^m m! \leq a^a \left(\frac{m}{a}\right)^m \leq a^{m/2} \left(\frac{m}{3}\right)^m \left(\frac{3}{a}\right)^m.
\]
But $(\frac{m}{3})^m \leq m!$ as soon as $m$ is large enough. So this yields $t^m \leq (\frac{9}{a})^{m/2}$ and the result follows.
\end{proof}

\begin{proof}[Proof of Lemma \ref{bochert}] By Lemma \ref{largeorb}, $H$ has an orbit $\Omega$ of size $m \geq \frac{d}{3}n$. The action of $H$ on $\Omega$ may not be primitive. Let $\Omega_0$ be a minimal block of imprimitivity, and $H_0$ the stabiliser of $\Omega_0$ in $H$. Let $K$ be the restriction of $H$ to $\Omega$ and $K_0$ the restriction of $H_0$ to $\Omega_0$. Note that $[Sym(\Omega):K] \leq [Sym(n):H] \leq d^{-n} \leq t^{-m}$. We may thus apply Lemma \ref{largeblock} to $K$ and get that the maximal number $a$ of blocks of imprimitivity is most $b=9t^{-2}$. In particular $[H:H_0] \leq b$. Similarly $[Sym(\Omega_0):K_0] \leq [Sym(n):H_0]\leq b d^{-n} \leq b t^{-m} \leq t^{-2m}$, and so $|K_0| \geq (\frac{m}{a})! t^{2m}$.

On the other hand, since the action of $K_0$ on $\Omega_0$ is primitive, Bochert's theorem \cite[Theorem 3.3B]{dixon-mortimer} implies that unless $K_0$ acts as $Sym(\Omega_0)$ or $Alt(\Omega_0)$ on $\Omega_0$ it satisfies $|K_0| \leq \frac{(m/a)!}{([(m/a+1)/2])!}$. Combining these two inequalities yields $(m/4a)^{m/4a} \leq t^{-2m}$, thus $m \leq 4a t^{-8a}$, and hence $n \leq \frac{12a}{d} t^{-8a}$. This implies that $n$ is bounded in terms of $\beta$ if $d \geq 1/(\log \log n)^{\beta}$.
\end{proof}

\bigskip

\subsection{Doubling close to the diameter and a question}
We end this section by asking a question.  Lemma \ref{lem:5adic.ph} shows that the almost flat condition implies doubling at some controlled scale. Moreover, doubling at some scale is enough to apply Theorem \ref{thm:bgt}, yielding the following.

\begin{prop}\label{doubling-structure} Consider a finite group $G$ with symmetric generating set $S$. Let $K\geq 1$. There is an integer $n_1=n_1(K)$ depending only on $K$ such that if $n$ is such that $n \geq n_1$ and
 $$|S^{2n+1}| \leq K |S^n|,$$
then $S^{8n}$ contains a normal subgroup $H$ of $G$ such that $G/H$ has a nilpotent subgroup whose index, rank and nilpotency class are bounded in terms of $K$ only.
\end{prop}

\begin{proof} By Lemma \ref{doublingreduce}, $|S^{5n}| \leq C(K) |S^n|$ for some constant $C(K)$ depending on $K$ only. Now apply Proposition \ref{prop:v.nilp}.
\end{proof}

Note that the quotient $G/H$ is non-trivial whenever $n < \frac{1}{8}\gamma$, where $\gamma$ is the diameter of $G$ with respect to $S$.
\bigskip

\noindent \emph{Question:} Assume that $|S^{2n+1}| \leq K |S^n|$ holds for some $n \in [\frac{1}{8}\gamma, (1-\eps)\gamma]$, say. Is it true that $G$ has a normal subgroup $H\ne G$ such that $G/H$ has a nilpotent subgroup whose index, rank and nilpotency class are bounded in terms of $K,\eps$ only?

\section{Cheeger constant and spectral gap}\label{expanderbounds}

In this section we prove Corollaries \ref{smallcheeger}, \ref{cheeger-gap}, \ref{cover-group} and \ref{cover-manifold}. We first recall some definitions.

Let $\mathcal G$ be a finite connected $k$-regular graph. Given a set $A\subset V(\mathcal G)$, denote by $\partial A$ the \emph{edge boundary} of $A$, defined as the set of edges between $A$ and $V(\mathcal G)\backslash A$.  The \emph{Cheeger constant} $h(\mathcal G)$ of a graph $\mathcal G$ is defined by
\[
h(\mathcal G)=\inf\left\{\frac{|\partial A|}{|A|}:A\subset V(\mathcal G),|A|\le\frac{1}{2}|\mathcal G|\right\}.
\]

Recall that $\mathcal G$ is called an \emph{$\eps$-expander} if $h(\mathcal G) \geq \eps$; see the survey \cite{exp.survey} for detailed background. Expansion may be considered in terms of the spectrum of the Laplacian operator on the graph. Specifically, the Laplacian $\Delta$ on our connected $k$-regular graph $\mathcal G$ is an operator on the finite-dimensional vector space $\ell^2(\mathcal G)$, and is defined by
\[
\Delta f(x)=k f(x)-\sum_{y\sim x}f(y),
\]
where $y \sim x$ means that $y$ and $x$ are connected by an edge in the graph. The Laplacian is self-adjoint on $\ell^2(\mathcal G)$, and so its eigenvalues
\[
\lambda_0(\mathcal G)=0<\lambda_1(\mathcal G)\le\ldots\le\lambda_{|\mathcal G|-1}(\mathcal G)
\]
are real. The link between this and expansion as defined above is given by the \emph{discrete Cheeger--Buser inequality},  which states that
\begin{equation}\label{eq:cheeger}
\frac{1}{2}h(\mathcal G)^2 \leq \lambda_1(\mathcal G)\leq 2k \cdot h(\mathcal G).
\end{equation}
See the book \cite{lubotzky-book} for a proof of these inequalities and the references therein for further historical background. Furthermore, we record the following well-known additional inequalities, which together with $(\ref{eq:cheeger})$ imply $(\ref{groupbounds})$.

\begin{lemma}[Diameter versus expansion]\label{lem:expand.diam}
For every finite $k$-regular graph $\mathcal G$ we have
\[
h(\mathcal G) \leq\frac{4k\log|\mathcal G|}{\diam(\mathcal G)}.
\]
If $\mathcal{G}$ is vertex transitive then
$$\frac{1}{2 \diam(\mathcal G)} \leq h(\mathcal G).$$
\end{lemma}
Lemma \ref{lem:expand.diam} is both standard (see e.g. \cite{dsc.diam} for a variant) and straightforward, but for completeness we present a short proof. Just as the proof of the Cheeger-Buser inequalities, this is based on the variational characterisation of $\lambda_1$ and $h$, which gives

\begin{equation}\label{varlam}
\lambda_1(\mathcal{G}) = \inf \left\{ \frac{\|\nabla f\|_2^2}{\|f\|_2^2} \, : \,\int f = 0\right\};
\end{equation}

\begin{equation}\label{varh}
h(\mathcal{G}) = \inf \left\{ \frac{\|\nabla f\|_1}{\inf_{m \in \R} \|f -m \|_1} \right\}.
\end{equation}
Here the infimums are taken over all real valued functions on the set of vertices of the graph. The norm $\|f\|_2$ is the $\ell^2$ norm, while $\|f\|_1$ denotes the $\ell^1$ norm, defined as usual by $\|f\|_1 = \sum_{x \in V(\mathcal{G})} |f(x)|$. Finally, $\nabla f$ is the non-negative function defined on the set of edges by $\nabla f (e)=|f(e^+) - f(e^-)|$, where $e^+$ and $e^-$ are the vertices of the edge $e$. The identity \eqref{varlam} may be found in \cite{lubotzky-book}. Since we do not have a readily available reference for \eqref{varh}, we provide a proof below. First let us use it to prove Lemma \ref{lem:expand.diam}.

\begin{proof}[Proof of Lemma \ref{lem:expand.diam}]
Abbreviate $h=h(\mathcal G)$, the Cheeger constant of $\mathcal G$. Given $x\in\mathcal G$ and $r>0$, write $B_r(x)$ for the closed ball of radius $r$ centred at $x$ for the graph metric on $\mathcal G$, and write $S_r(x)=B_r(x)\backslash B_{r-1}(x)$ for the sphere of radius $r$ centred at $x$.

By definition of $h$, if $|B_r(x)|\le\frac{1}{2}|\mathcal G|$ then we have $|\partial B_r(x)|\ge h|B_r(x)|$. Since each vertex of $S_{r+1}(x)$ is incident with at most $k$ edges in $\partial B_r(x)$ we have $|S_{r+1}(x)|\ge k^{-1}|\partial B_r(x)|$, which implies that
\[
|B_{r+1}(x)|\ge\left(1+\frac{h}{k}\right)|B_r(x)|.
\]
Noting that $\frac{h}{k}\le1$, and hence that $1+\frac{h}{k} \geq e^{\frac{h}{2k}}$, it follows immediately that $|B_r(x)|\ge\min\left\{\frac{1}{2}|\mathcal G|,e^{rh/2k}\right\}$. In particular, if $r\frac{h}{k}\ge2\log|\mathcal G|$ then the balls $B_r(x)$ and $B_r(y)$ have size at least $\frac{1}{2}|\mathcal G|$, and hence non-trivial intersection for every choice of $x$ and $y$. This proves the first inequality of the lemma.

Now assume that $\mathcal{G}$ is vertex transitive. Set $\gamma:=\diam(\mathcal{G})$, and write $G=\Aut(\mathcal{G})$, the (finite) group of automorphisms of the graph. Pick a base vertex $x_0$ in the graph and connect every other vertex $y$ to $x_0$ by a geodesic path, say $x_0=y_0,y_1,\ldots,y_n=y$, noting that $n\le\gamma$. For any real-valued function $f$ on the vertices of $\mathcal G$ we have
\[
\sum_{\sigma \in G} |f(\sigma(x_0)) - f(\sigma(y))| \leq \sum_1^n \sum_{\sigma \in G} |f(\sigma(y_{i-1})) - f(\sigma(y_i))| \le \|\nabla f\|_1 \sum_1^n |G_{e_i}|,
\]
where $G_{e_i}$ is the stabiliser of the edge $e_i$ between $y_{i-1}$ and $y_i$. Clearly $G_{e_i}$ has  a subgroup of index at most $2$ that fixes $y_i$, and hence $|G_{e_i}| \leq 2 |G|/|\mathcal{G}|$. Summing over all vertices $y$ of $\mathcal G$, it follows that
\[
\inf_{m\in\R} ||f-m||_1 \leq \frac{1}{|G|}\sum_{\sigma \in G} \|f - f(\sigma(x_0))\|_1
=\frac{1}{|G|}\sum_{\sigma \in G} \sum_{y \in \mathcal{G}} |f(\sigma(x_0)) - f(\sigma(y))| \leq  2\gamma\|\nabla f\|_1
\]
and so the lemma follows from $(\ref{varh})$.
\end{proof}

\begin{proof}[Proof of \eqref{varh}]
To see that $h(\mathcal G)$ dominates the right hand side, we simply pick the function $f:=1_{A}- 1_{\mathcal G \setminus A}$, where $A$ is a subset of vertices of size at most $|\mathcal G|/2$ realising the infimum in the definition of $h(\mathcal G)$. Then note that $\|\nabla f\|_1=2|\partial A|$ and that $\|f-m\|_1$ is minimised by taking $m=-1$, which makes its value equal to $2|A|$.
Now to verify the opposite inequality, given a function $f$ on the vertices consider the subsets of vertices $\Omega_t:=\{x:f(x)\geq t\}$ and let $m:=\inf\{t \in \R ;  |\Omega_t|\leq |\mathcal G|/2\}$. Note that if $t>m$, then $|\Omega_t|\leq |\mathcal G|/2$, while if $t\leq m$, then $|\mathcal G \setminus \Omega_t|\leq |\mathcal G|/2$. On the other hand, $\|f-m\|_1=\sum_{x \in \mathcal G}[ (f-m)^+(x) + (m-f)^+(x)]$, where $u^+=\max\{0,u\}$ is the positive part of the real number $u$. Hence
\[
\|f-m\|_1 = \int_{-\infty}^m |\mathcal G \setminus \Omega_t|dt + \int_m^{+\infty} |\Omega_t|dt \leq \frac{1}{h} \int_{-\infty}^{+\infty} |\partial \Omega_t|dt,
\]
since $\partial \Omega_t=\partial (\mathcal G \setminus \Omega_t)$. On the other hand, observe that for every edge $e$ we have $\nabla f(e) = \int_{-\infty}^{+\infty} 1_{\partial \Omega_t}(e) dt$, and so $\|\nabla f \|_1=\int_{-\infty}^{+\infty} |\partial \Omega_t|dt$.  The inequality follows.
\end{proof}

We now pass to the proofs of Corollaries \ref{smallcheeger} and \ref{cheeger-gap}.

\begin{proof}[Proof of Corollary \ref{smallcheeger}]
We have $\gamma^2 \geq \frac{1}{8\lambda_1}$ by the left hand side of $(\ref{groupbounds})$, and so if $\lambda_1 \leq 2^{-9}$ then $\gamma \geq 8$, and hence $\gamma^3\ge1/\lambda_1$. 
On the other hand, $\gamma \geq 1/2h$ by Lemma \ref{lem:expand.diam}, so if $h\leq\frac{1}{4}$ then $\gamma \geq 2$, and hence $\gamma^2 \geq 1/h$.
\end{proof}

To prove Corollary \ref{cheeger-gap} we use the following general lemma.

\begin{lemma}\label{lambdadoubling} Let $\mathcal G$ be the Cayley graph of a finite group $G$ with symmetric generating set $S$ and diameter $\gamma$. Then
$$\lambda_1(\mathcal G) \leq \frac{9|S|}{\gamma^2} \frac{|G|}{|S^{\gamma/3}|}$$
\end{lemma}

\begin{proof} Pick $a,b \in G$ such that $d(a,b)=\gamma$, and let $f(x):=d(x,a)-d(x,b)$, where $d(x,y)$ is the graph distance in $\mathcal G$. Note that $f$ has zero mean and that $|\nabla f(e)| \leq 2$ for every edge $e$. Hence $\|\nabla f\|_2^2 \leq 2|S||G|$. On the other hand $|f(x)|\geq \gamma/3$ on the balls of radius $\gamma/3$ around $a$ and around $b$. Hence $\|f\|_2^2 \geq 2(\gamma/3)^2 |S^{\gamma/3}|$. The result therefore follows from $(\ref{varlam})$.
\end{proof}

\begin{proof}[Proof of Corollary \ref{cheeger-gap}] Set $K=5^{4/\eps}$, and let $n_0=n_0(K)$ be the constant appearing in the conclusion of Theorem \ref{thm:scales}. Provided $\gamma\ge n_0^4$, Lemma \ref{lem:5adic.ph} implies that there is $n$ satisfying $n_0\le n\le\gamma^{1/2}$ such that $|S^{5n}|\le|S^n|$. Theorem \ref{thm:scales} therefore implies that $|G|=|S^{\gamma}|\le O_K(1)|S^{\gamma/3}|$. The bound \eqref{eq:ch-gap.2} therefore follows from Lemma \ref{lambdadoubling}, and the bounds \eqref{eq:ch-gap.1} then follow from \eqref{groupbounds} and \eqref{eq:ch-gap.2}.
\end{proof}

\begin{remark}[Cheeger constant versus $\lambda_1$] We have shown that almost flat groups satisfy the strong form of the Buser inequality; that is to say $\lambda_1$ is comparable to the square of the Cheeger constant. This property does not characterise almost flatness. Indeed it is shared for example by the Cayley graph of the lamplighter group based on a finite cyclic group (see \cite{peres-revelle}), which is by no means almost flat. By contrast, there are Cayley graphs of finite groups at the opposite extreme permitted by \eqref{groupbounds}, which is to say $\lambda_1$ is comparable to the Cheeger constant itself, not its square. This happens, for example, for the lamplighter group based on $\Z/n^d\Z$ for any given $d \geq 3$. For all $d\geq 1$ the Cheeger constant is comparable to $1/\gamma$ (almost extremal sets are obtained by setting one fixed light off), but $\lambda_1$ is comparable to $1/\gamma^2$ if $d=1$, to $1/\gamma \log \gamma$ if $d=2$, and to $1/\gamma$ if $d \geq 3$ \cite{peres-revelle}.
\end{remark}

\begin{remark}[Expansion and number of generators] Almost flat groups are not expanders, since their $\lambda_1$ values are comparable to $1/\gamma^2$ by Corollary \ref{cheeger-gap}. In fact, in order to construct a Cayley graph of a finite group $G$ with a nilpotent subgroup of bounded index, rank and nilpotency class with a positive lower bound on $\lambda_1$, one would need at least a constant times $\log |G|$ generators. This follows immediately from Lemma \ref{lem:ab.quotient} and the corresponding result for abelian groups, which is proved in \cite[Propositions 3.9 \& 3.10]{lubotzky-weiss}. Note that this is extremal in some sense, because according to Alon and Roichman \cite[Theorem 1]{ar}, given $\eps >0$, every finite group $G$ has a generating set $S$ of size at most $C_\eps\log|G|$ with $\eps$-expanding generating set.
\end{remark}

We now pass to the proof of Corollaries \ref{cover-group} and \ref{cover-manifold}.

\begin{proof}[Proof of Corollary \ref{cover-group}] We fix a finite symmetric generating set for $\Gamma$ and consider the Cayley graphs of the quotients $\Gamma/\Gamma_n$ with respect to this generating set. The second inequality in Lemma \ref{lem:expand.diam} implies that $\diam (\Gamma/\Gamma_n) \geq \frac{1}{2}|\Gamma/\Gamma_n|^\eps$. Being finitely generated, $\Gamma$ has only finitely many subgroups of index at most $M$ for any given $M\geq 1$, and so it follows that $\Gamma/\Gamma_n$ is $\frac{\eps}{2}$-almost flat as soon as $n$ is large enough. It then follows from Theorem \ref{thm:gromovfinite} that $\Gamma/\Gamma_n$ has a quotient of diameter at least $\frac{1}{2}\diam(\Gamma/\Gamma_n)$ that has an abelian subgroup whose index is bounded independently of $n$. This implies that there is a finite-index subgroup $H \leq \Gamma$ that maps onto abelian groups of unbounded order, which in particular means that $H/[H,H]$ is infinite. Since $H/[H,H]$ is also finitely generated and abelian, it admits a surjective homomorphism onto $\Z$.
\end{proof}

The proof of Corollary \ref{cover-manifold}, the manifold version of Corollary \ref{cover-group}, is a simple application of the Brooks-Burger transfer principle, for which we refer the reader to the paper \cite{mantuano} and the appendix to the lecture notes \cite{breuillard-pcmi}, as well as the original sources \cite{brooks,burger}. This states that given a compact connected Riemannian manifold $M$, and a finite symmetric generating set $S$ of its fundamental group, there are positive constants $c_1,c_2$ such that for every finite-degree Riemannian cover $M'\to M$, we have
\begin{equation}\label{eq:lambda.fund.grp}
c_1 \lambda_1(M')\leq \lambda_1(\pi_1(M)/\pi_1(M'),S) \leq c_2 \lambda_1(M')
\end{equation}
where the $\lambda_1(M')$ denotes the first non-zero eigenvalue of the Laplace-Beltrami operator on $M'$ and $\lambda_1(\pi_1(M)/\pi_1(M'),S)$ that of the corresponding Schreier graph. Similarly,
\begin{equation}\label{eq:cheeger.fund.grp}
c_1 h(M')\leq h(\pi_1(M)/\pi_1(M'),S) \leq c_2 h(M').
\end{equation}

It is worth emphasising that the constants $c_1,c_2$ here depend only on the base manifold $M$ and choice of $S$, but not on $M'$.

\begin{proof}[Proof of Corollary \ref{cover-manifold}] This is just the combination of $(\ref{eq:cheeger.fund.grp})$ and Corollary \ref{cover-group}.
\end{proof}

\section{Moderate growth and mixing times}\label{sec:mod.growth}

In this section we prove Corollary \ref{cor:diac.sc}, which shows that almost flat groups and groups with moderate growth coincide. We then discuss mixing-time estimates for random walks on almost flat groups, and prove Corollary \ref{cor:mixing} and Theorem \ref{mix}.

We start with a crude, yet general, bound on the diameter of a Cayley graph (which is meaningful only when $S$ is large, since otherwise it is worse than the trivial bound $\diam_S(G) \leq |G|$).

\begin{lemma}\label{freimanlem} Let $G$ be a finite group and $S$ a symmetric generating set. Then $$\diam_S (G) \leq 2(|G|/|S|)^\frac{7}{4}.$$
\end{lemma}

\begin{proof}
Let $m$ be the least integer such that $|S^{2^{m+1}}|\le\frac{3}{2}|S^{2^m}|$, noting that this implies that $|G|/|S|\ge\left(\frac{3}{2}\right)^m$. Freiman's lemma \cite{freiman-PAMS} shows that the bound $|S^{2^{m+1}}|\le\frac{3}{2}|S^{2^m}|$ implies that $S^{2^m}$ is contained in a coset of a subgroup $H\le G$ satisfying $|H| \leq \frac{3}{2}|S^{2^m}|$. Since $1\in S$, we in fact have $S^{2^m}\subset H$, and since $S$ generates $G$ this implies that $H=G$, and in particular that $|S^{2^m}|\ge\frac{2}{3}|G|$. This in turn implies that $S^{2.2^m}=G$, and hence that $\diam_S(G)\le2.2^m$. The result follows since $\left(\frac{3}{2}\right)^\frac{7}{4}\ge2$.
\end{proof}

\begin{proof}[Proof of Corollary \ref{cor:diac.sc}] In the first direction, suppose that $(G,S)$ has $(A,d)$-moderate growth, which is to say that $(\ref{moderatedef})$ holds for all $n\geq 1$. The case $n=1$ implies that $\gamma \geq A^{-1/d} (|G|/|S|)^{1/d}$. However, $\gamma \leq 2(|G|/|S|)^2$ by Lemma \ref{freimanlem}, so $(G,S)$ is $\frac{1}{2d}$-almost flat whenever $\gamma \geq 2A^4$.

In the converse direction, suppose that $(G,S)$ is $\eps$-almost flat, let $K=5^{\frac{4}{\eps}}$, and let $n_0=n_0(K)$ be the integer given by Theorem \ref{thm:scales}. If $\gamma<n_0^4$ then the almost flat assumption implies that for every $n\le\gamma$ we have
\[
|S^n|\ge|S|\ge\left(\frac{n}{n_0^4}\right)^{\varepsilon^{-1}}|S|\ge\frac{1}{n_0^{4\varepsilon^{-1}}}\left(\frac{n}{\gamma}\right)^{\varepsilon^{-1}}|G|,
\]
and so $G$ has $(n_0^{4\varepsilon^{-1}},\varepsilon^{-1})$-moderate growth. We may therefore assume that $\gamma\ge n_0^4$. Lemma \ref{lem:5adic.ph} therefore implies that there is some $n$ satisfying $n_0\le n\le\gamma^{1/2}$ such that $|S^{5n}|\leq K |S^{n}|$, and so Theorem \ref{thm:scales} implies that $|S^{5n}|\leq \theta(K)^5 |S^n|$ for all $n \geq \gamma^{1/2}$. Given $n\ge\gamma^{1/2}$, therefore, writing $r$ for the smallest integer such that $5^r n \geq \gamma$ we have $|S^{5^r n}| \leq \theta(K)^{5r} |S^{n}|$ and $\theta(K)^{5r} \leq \theta(K)^{5+ 5\log_5(\gamma/n)}=\theta(K)^{5}(\gamma/n)^{5 \log_5(\theta(K))}$, and so
\begin{equation}\label{eq:mg}
|S^n| \geq \frac{1}{A}\left(\frac{n}{\gamma}\right)^d|G|
\end{equation}
holds with $A=\theta(K)^5$ and $d=5\log_5 \theta(K)$. When $n \leq \gamma^{1/2}$, on the other hand, the almost flat assumption implies that $|S|/|G|\ge1/\gamma^{\varepsilon^{-1}}$, and so
\[
\frac{|S^n|}{|G|}\ge\frac{|S|}{|G|}\ge\frac{1}{\gamma^{\varepsilon^{-1}}}
=\left(\frac{\gamma^{1/2}}{\gamma}\right)^{2\varepsilon^{-1}}\ge\left(\frac{n}{\gamma}\right)^{2\varepsilon^{-1}}.
\]
Thus \eqref{eq:mg} holds with $d=2\varepsilon^{-1}$ and $A=1$.
\end{proof}

The main result of the work \cite{dsc} of Diaconis and Saloff-Coste shows that groups with moderate growth have mixing time quadratic in the diameter. Inspired by this, we spend the rest of this section discussing various other sufficient conditions on the Cayley graph of a finite group for its mixing time to be quadratic in the diameter.

We should emphasise that the constants in the results of Diaconis and Saloff-Coste are effective in the moderate-growth parameters of the group in question; the ineffectiveness of the constants appearing in the results of this section arise only due to our reliance on Theorem \ref{thm:bgt}.

\bigskip

By definition, the $L^p$ mixing time $T_p$ is the first integer $n$ such that $\|\mu_S^n - \mu_G\|_p \leq \frac{1}{10} \|\mu_G\|_p$, where $\mu_S$ is the uniform probability measure on the generating set $S$ and $\mu_G$ the uniform probability measure on $G$. We view them as functions on $G$ taking non-negative values. Of course, the choice of factor $\frac{1}{10}$ is purely a matter of convention. It is customary also to define the relaxation time $T_{rel}$ as the inverse spectral gap, or more specifically $T_{rel}=1/(1-\beta_S)$, where $\beta_S$ is the norm of the operator $1-\Delta/|S|$. For simplicity, we will assume that $\beta_S=1-\lambda_1/|S|$, which amounts to assuming that $\Delta$ has no eigenvalue too close to $2|S|$. This is the case if $\lambda_1 \leq 2$ and $1 \in S$.

The mixing times $T_p$ dominate the relaxation time $T_{rel}$ and are non-decreasing in $p$, and in fact there are essentially only two mixing times, $T_1$ and $T_\infty$, because $T_p$ is comparable to $T_\infty$ if $p>1$. The example of the lamplighter group with cyclic base shows that $T_\infty$ really can be much larger than $T_1$ \cite{haggstrom-jonasson,peres-revelle}. Writing $\gamma=\diam_S(G)$, we have $\frac{1}{2}\diam_S(G)\le T_p\le O_{|S|}(\gamma^3)$, and the same lamplighter example shows that the upper bound can indeed be reached for $T_\infty$ \cite{haggstrom-jonasson,peres-revelle}. It is not known, however, whether $T_1$ is always $O_{|S|}(\gamma^2)$. We summarise these facts and some others in the following proposition, which is well known and encapsulates the basic properties of mixing times on Cayley graphs.


\begin{prop}[Basic facts on mixing times]\label{basicmixing} Let $G$ be a finite group with symmetric generating set $S$ containing $1$ and denote by $\gamma$ the diameter of the Cayley graph. Assume that $\lambda_1\leq 2$, and set $\beta_S:=1-\lambda_1/|S|$. Then, for $p \in [1,\infty]$ and $n\in\N$, the following conditions hold.
\begin{enumerate}
\item \label{mix1}The quantity $\|\mu_S^{(n)} -\mu_G\|_p$ is non-increasing in $n$.
\item \label{mix2}The quantity $\|\mu_S^{(n)} -\mu_G\|_p/\|\mu_G\|_p$ is non-decreasing in $p$, and hence so is $T_p$.
\item \label{basicin-lower} We have $\beta_S^n  \leq   \|\mu_S^{(n)} -\mu_G\|_p/\|\mu_G\|_p$.
\item \label{basicin-upper} We have
\[
\frac{\|\mu_S^{(2n)} -\mu_G\|_p}{\|\mu_G\|_p } \leq \left( \frac{ \|\mu_S^{(n)} -\mu_G\|_2}{\|\mu_G\|_2} \right)^2.
\]
\item \label{basicin2}We have $\|\mu_S^{(n)} -\mu_G\|_2 \leq \beta_S^n$.
\item \label{mix6}For every $p>1$ there exists $C(p)>0$ such that $T_\infty\le C(p)T_p$. For $p\ge2$ we may take $C(p)=2$.
\item \label{mix.half.diam}We have $T_p\ge\frac{1}{2}\gamma$ for all $p$, and $T_\infty\ge\gamma$.
\item \label{mix7}We have $T_2\leq 8|S|\gamma^2 \log |G| \leq 8|S|\log |S| \gamma^3$.
\item \label{mix8}We have $T_{rel}\le\min\{T_1,8|S|\gamma^2\}$.
\end{enumerate}
\end{prop}

We include the straightforward proof for the reader's convenience.

\begin{proof}
Condition \eqref{mix1} holds with $\|\,\cdot\,\|_p$ replaced by any $G$-invariant norm $\|\,\cdot\,\|$, since the triangle inequality implies that $\|\mu_S*f\|\le\|f\|$ for every $f:G\to\R$. To prove \eqref{mix2}, assume that $p<p'<\infty$ and note that the function $u\mapsto u^{p'/p}$ is convex, and hence that we may  apply Jensen's inequality in the form $\mu_G(f^{p'/p}) \geq (\mu_G(f))^{p'/p}$ to the function $f(x)=\mu_S^{(n)}(x)-\mu_G(x)$, from which \eqref{mix2} follows (the case $p'=\infty$ being a trivial calculation).

For the bound \eqref{basicin-lower}, let $q$ be such that $1/p+1/q=1$, so that by duality we have
\begin{equation}\label{dual}
\|\mu_S^{(n)} - \mu_G\|_p=\sup_{\|f\|_q\leq 1} |\mu_S^n(f)-\mu_G(f)|.
\end{equation}
Let $f$ be an eigenfunction of the Laplacian with eigenvalue $\lambda_1$, translated so that $f(1)=\|f\|_\infty$ and scaled so that $\|f\|_q=1$. Note that $\mu_G(f)=0$, and also that $\mu_S^{(n)}*f = \beta_S^n f$, and hence that $\mu_S^{(n)}(f) = \beta_S^n f(1)=\beta_S^n\|f\|_\infty$. It therefore follows from \eqref{dual} that $\|\mu_S^{(n)} - \mu_G\|_p\ge\beta_S^n\|f\|_\infty\ge\beta_S^n\|f\|_q|G|^{-1/q}=\beta_S^n\|\mu_G\|_p$, and so \eqref{basicin-lower} follows.

It is enough to prove \eqref{basicin-upper} for $p=\infty$, but this follows from the Cauchy-Schwarz inequality applied to the obvious identity $(\mu_S^{(2n)} -\mu_G )(x) =  (\mu_S -\mu_G)^{(2n)}(x) = \sum_y (\mu_S -\mu_G)^{(n)}(xy^{-1})(\mu_S -\mu_G)^{(n)}(y)$, which is valid for every $x \in G$, and yields
$\|\mu_S^{(2n)} -\mu_G\|_\infty \leq \|\mu_S^{(n)} -\mu_G (x)\|_2^2$, and hence \eqref{basicin-upper} after observing that $\|\mu_G\|_\infty=\|\mu_G\|_2^2$.

To see $(\ref{basicin2})$, write $P_{\mu_S}$ for the operator of left convolution by $\mu_S$, and note that \eqref{dual} gives
\begin{equation}\label{l2}
\|\mu_S^{(n)} - \mu_G\|_2 = \sup_{\|f\|_2\leq 1, \mu_G(f)=0} |P_{\mu_S}^n f (1)|\leq \|P_{\mu_S}^n\|_2 =\beta_S^n.
\end{equation}

It is trivial that $T_\infty\ge\gamma$, so to prove \eqref{mix.half.diam} it suffices to check that $T_1>\frac{1}{2}\gamma$. To see this, note that the ball of radius $T_1$ contains at least half of the elements of $G$, since otherwise $\|\mu_S^{T_1}-\mu_G\|_1\ge\frac{1}{2}$, and so the ball of radius $2T_1$ is all of $G$.

The fact that \eqref{mix6} holds with $C(p)=2$ for $p\ge2$ follows from \eqref{mix2} and \eqref{basicin-upper}. When $p\in(1,2)$, applying Young's inequality to the identity $(\mu_S^{(2n)} -\mu_G ) =  (\mu_S^{(n)} -\mu_G)^{(2)}$ implies that $\|\mu_S^{(2n)} -\mu_G \|_{\frac{p}{2-p}} \leq \|\mu_S^{(n)} -\mu_G\|_p^2$, and hence that $T_{\frac{p}{2-p}} \leq 2T_p$. Iterating this a finite number of times if necessary, we see that $T_2\leq C'(p) T_p$ for some constant $C'(p)>0$, and \eqref{mix6} follows.

The first inequality of \eqref{mix7} follows immediately from $(\ref{basicin2})$ and the lower bound on $\lambda_1$ from $(\ref{groupbounds})$, while the second inequality follows from the fact that $|G|=|S^\gamma|\leq |S|^\gamma$. Finally, to prove \eqref{mix8}, note that \eqref{groupbounds} shows that $T_{rel} \leq 8|S|\gamma^2$, while the fact that $T_{rel}\le T_1$ follows from \eqref{basicin-lower} and the assumption that $\lambda_1\le2$.
\end{proof}

\begin{proof}[Proof of Corollary \ref{cor:mixing}]  The lower bound follows immediately from $(\ref{basicin-lower})$ and the bound \eqref{eq:ch-gap.2} from Corollary \ref{cheeger-gap}. For the upper bound, we need to improve Proposition \ref{basicmixing} \eqref{basicin2} into $ \|\mu_S^{(n)} -\mu_G\|_2 \leq B \beta_S^n \|\mu_G\|_2$. Getting this extra factor $\|\mu_G\|_2$ is the heart of the matter, and it was established by Diaconis and Saloff-Coste in \cite[Theorem 3.2]{dsc} for groups of moderate growth with an explicit dependence of the constant $B$ in terms of the moderate growth parameters $A$ and $d$. Corollary \ref{cor:diac.sc} implies that $G$ has such parameters depending only on $\eps$, and so the desired bound follows. (The reader should note that, while \cite[Theorem 3.2]{dsc} is stated for $\|\mu_S^{(n)} -\mu_G\|_1$, it also holds for $\|\mu_S^{(n)} -\mu_G\|_2/\|\mu_G\|_2$ by Remark 2 on page 6 of \cite{dsc}.)
\end{proof}

We now discuss and prove Theorem \ref{mix}. Almost flat groups or groups of moderate growth are doubling at at least one scale $\leq \gamma^{\delta}$ for any fixed $\delta>0$ by Lemma \ref{lem:5adic.ph}, and it is legitimate to ask whether this property alone is enough to give a mixing time quadratic in $\gamma$. The following reformulation of Theorem \ref{mix} shows that if a Cayley graph is doubling at one scale $\leq \gamma^{2/3}$ then this is indeed enough to guarantee quadratic mixing time, thus extending the validity of the Diaconis--Saloff-Coste theorem \cite{dsc}.

\begin{theorem}[Doubling and quadratic mixing time]\label{doubling-quadratic}  Let $G$ be a finite group with symmetric generating set $S$ and denote by $\gamma$ the diameter of the Cayley graph. Let $K\geq 1$ be a parameter. Then there is a constant $C_K>0$ depending on $K$ only such that if $|S^{2n+1}|\leq K |S^n|$ for some $n\geq C_K$ the
 \begin{itemize}
 \item if $n\leq  \gamma/3$, then $T_{rel}$ is quadratic in the sense that $T_{rel} \geq \gamma^2 /C_K$;
 \item if $n \leq \gamma^{2/3}$, then $T_p$ is also quadratic for each $p \in [1,\infty]$, in the sense that $T_p = O_{K,|S|}(\gamma^2)$.
 \end{itemize}
\end{theorem}

In proving Theorem \ref{doubling-quadratic} we make use of the following variant of Lemma \ref{lem:expand.diam}.

\begin{lemma} \label{lambda1quotient} Let $\mathcal{G}$ be the Cayley graph of a finite group $G$ with symmetric generating set $S$. Let $H \leq G$ be a normal subgroup and let $\gamma_H$ be the diameter of $H$ in the graph distance. Assume that $f$ is a function on $G$ such that $\sum_{x \in gH} f(x) = 0$ for every $g \in G$. Then
$$\frac{\|\nabla f\|_2^2}{\|f\|_2^2} \geq \frac{1}{\gamma_H^2}$$
\end{lemma}

\begin{proof} Pick $y\in H$ and connect it to the identity by a geodesic path $1=y_0,y_1,\ldots,y_n=y$, where $y_{i}=s_iy_{i-1}$ for some $s_i \in S$, and $n \leq \gamma_H$. Observe that for each $i$ we have $\sum_{g \in G} |f(y_{i-1}g) - f(y_{i}g)|^2 \leq \|\nabla f\|_2^2$. The Cauchy--Schwarz inequality in the form $\left(\sum_{i=1}^n|v_i|\right)^2\le n\|v\|_2^2$ then gives
\begin{equation}\label{l1q.1}
\sum_{g \in G} |f(g) - f(yg)|^2 \leq  \sum_{g \in G} \gamma_H \sum_{1}^n |f(y_{i-1}g) - f(y_ig)|^2\leq \gamma_H^2 \|\nabla f\|_2^2.
\end{equation}
On the other hand, for $g \in G$ we have $\sum_{y \in H} |f(g) - f(yg)|^2 = \|f - f(g)\|^2_{L^2(gH)}$ since $gH=Hg$, and so
\begin{equation}\label{l1q.2}
\sum_{y \in H} |f(g) - f(yg)|^2 \geq \|f\|^2_{L^2(gH)}
\end{equation}
since $f$ has mean zero on $gH$. Summing \eqref{l1q.1} over $y\in H$ and \eqref{l1q.2} over $g \in G$ shows that $|H|\|f\|_2^2\le\sum_{g \in G}\sum_{y\in H} |f(g) - f(yg)|^2\le|H|\gamma_H^2\|\nabla f\|_2^2$, and so the lemma is proved.
\end{proof}

\begin{proof}[Proof of Theorem \ref{doubling-quadratic}] The first item follows directly from the combination of Theorem \ref{thm:scales} and Lemma \ref{lambdadoubling}.

For the second item, we first note that if $C_K$ is chosen large enough then Proposition \ref{doubling-structure} implies that there is a normal subgroup $H \leq G$ contained in $S^{8\gamma^{2/3}}$ such that $G/H$ has a nilpotent subgroup whose index, rank and nilpotency class are bounded in terms of $K$ only. Note that the diameter $\gamma_{G/H}$ of the quotient Cayley graph is at least $\gamma/2$ (unless $\gamma$ is bounded), and hence that, by Proposition \ref{prop:ab.converse} and Lemma \ref{lem:rei-schr}, this quotient Cayley graph is $\eps$-almost flat for some $\eps$ depending only on $|S|$ and $K$. Corollary \ref{cor:mixing} therefore applies, and we conclude (denoting by $P:=P_{\mu_S}$ the operator of left convolution by $\mu_S$ and recalling ($\ref{l2}$)) that if $u$ is a zero mean function in $\ell^2(G/H)$, then $\|P^n u \|_2 \leq Be^{-n/C\gamma^2} \|u\|_2 \|\mu_{G/H}\|_2$ for all $n \geq 1$ and for some positive constants $B,C$ depending only on $|S|$ and $K$.

Now we decompose the space $\ell^2(G)$ as a direct sum $\ell^2(G)=\ell^2(G/H) \oplus \ell^2(G/H)^{\perp}$, where $\ell^2(G/H)$ is this time viewed as the subspace of $H$-invariant functions and $\ell^2(G/H)^{\perp}$ its orthogonal, the subspace of functions with zero mean on each coset of $H$. Both spaces are $G$-invariant, and given a function $f$ on $G$ we may write $f$ uniquely as $f=u+v$, where $u \in \ell^2(G/H)$ and  $v \in \ell^2(G/H)^{\perp}$, and $\|f\|_2^2 = \|u\|_2^2+\|v\|_2^2$. In particular, this implies that $\|P^n f\|_2^2 = \|P^n u\|_2^2 +\|P^n v\|_2^2$. However, the norm of $P$ in restriction to $\ell^2(G/H)^{\perp}$ is given by $1-\lambda/|S|\le e^{-\lambda/|S|}$, where $\lambda$ is the smallest eigenvalue of $\Delta$ in restriction to $\ell^2(G/H)^{\perp}$. Lemma \ref{lambda1quotient} shows that $\lambda\ge1/\gamma_H^2$, where $\gamma_H$ is the diameter of $H$ in $G$, and so we have $\|P^n v\|_2 \leq e^{-\frac{n}{|S|\gamma_H^2}}\|v\|_2$.

Now note that $\mu_S^{(n)} - \mu_G = P^n(f)$, where $f=1_{e} - \mu_G$. Decompose $f$ as $f=\overline{f} + (f-\overline{f})$, where $\overline{f}$ is the $H$-invariant function $\overline{f}(g):=\frac{1}{|H|}\sum_{x \in gH} f(x)$, so that $f-\overline{f}$ has zero mean on $H$-cosets, and note therefore that

$$\|\mu_S^{(n)} - \mu_G\|_2^2 = \|P^n(\overline{f})\|_2^2 + \|P^n(f-\overline{f})\|_2^2  \leq B^2e^{-2n/C\gamma^2} \|\overline{f}\|_2^2 \|\mu_{G/H}\|_2^2 +  e^{-\frac{2n}{|S|\gamma_H^2}}\|f-\overline{f}\|_2^2$$

Finally, observe that $\overline{f}=\frac{1}{|H|} 1_H - \mu_G$, that $\|\overline{f}\|_2^2 = 1/|H| - 1/|G| \leq 1/|H|=\|\mu_G\|_2^2/\|\mu_{G/H}\|_2^2$, and that $\|f-\overline{f}\|_2^2=1-\frac{1}{|H|} \leq 1$. Hence
$$\|\mu_S^{(n)} - \mu_G\|_2^2 \leq  B^2e^{-2n/C\gamma^2} \|\mu_G\|_2^2 + e^{-\frac{n}{|S|\gamma_H^2}}.$$

Now note that $\gamma_H^2 \log |G| = \gamma_H^2 \log (|G|/|H|) + \gamma_H^2 \log |H|$. But $G/H$ is $\eps$-almost flat, so $|G/H| \leq |S| \gamma^{1/\eps}$, and clearly $|H| \leq |S|^{\gamma_H}$, so $\log |H| \leq \gamma_H \log |S|$. On the other hand $H$ is contained in the ball of radius $8\gamma^{2/3}$ so $\gamma_H \leq 8\gamma^{2/3}$. It follows that
$$\gamma_H^2 \log |G| \ll_{K,|S|} \gamma_H^{2} \log \gamma + \gamma_H^3 \ll \gamma^2,$$
and therefore that $\|\mu_S^{(n)} - \mu_G\|_2 \leq \frac{1}{10} \|\mu_G\|_2$ as soon as $n \gg_{K,|S|} \gamma^2$. This gives the result for $T_2$, and hence for all $T_p$ by Proposition \ref{basicmixing}.
\end{proof}

We finish by presenting a simple example showing that the exponent $\frac{2}{3}$ in the second item of Theorem \ref{mix} is sharp. We thank T. Zheng for suggesting it.

\begin{example}[Sharpness of the $\frac{2}{3}$ in Theorem \ref{mix}]\label{sharpex} Let $L_M$ be the lamplighter group $L_M:=\Z/M\Z \ltimes B$, where $B$ is the direct product of $M$ copies of $\Z/2\Z$ that are cyclically permuted by the action of $\Z/M\Z$, and let $G:=L_M \times \Z/N\Z$. Let $S:=S_1 \times S_2$, where $S_1$ is the standard generating set of $L_M$ (consisting of the identity, a move of one space to the left or right, or a switch of the current lamp), and $S_2$ is the standard generating set $\{0, \pm 1\}$ of $\Z/N\Z$. Clearly the diameter $\gamma$ of $G$ is comparable to $\max \{M,N\}$. On the other hand, the uniform mixing time $T_\infty(G)$ is at least $\max\{T_\infty(L_M), T_\infty(\Z/N\Z)\}$, which is comparable to $\max\{M^3,N^2\}$ (see \cite{haggstrom-jonasson}). Let $\alpha \in (\frac{2}{3},1)$ and set $M:=N^{\alpha}$. Then $\gamma \simeq N$, and $T_\infty(G)$ grows faster than $N^2\simeq\gamma^2$. However, $M\simeq\gamma^\alpha$, and we have uniform doubling at all scales $\gtrsim M$.
\end{example}

\end{document}